\documentclass[11pt,reqno]{amsart}
\usepackage{eurosym}
\usepackage{amsmath,amstext,amssymb,amscd,hyperref}
\usepackage{verbatim}
\usepackage{enumerate}
\usepackage{mathrsfs}
\usepackage[dvipsnames,usenames]{xcolor}
\usepackage{amsthm}
\newtheorem{theorem}{Theorem}[section]
\newtheorem{lemma}[theorem]{Lemma}
\newtheorem{corollary}[theorem]{Corollary}
\newtheorem{proposition}[theorem]{Proposition}

\theoremstyle{definition}
\newtheorem{definition}[theorem]{Definition}

\newtheorem{remark}[theorem]{Remark}
\numberwithin{equation}{section}
\title[Inertial manifolds for the nonlocal parabolic problem]{Inertial manifolds for the nonlocal parabolic problem}
\author[X. Q. Yang, A. N. Carvalho, and C.Y. Sun ]{Xiaoqing Yang ${}^\dagger$,~Alexandre N. Carvalho ${}^\ddag$,~ and Chunyou Sun${}^{*,\dagger}$}
\thanks{Corresponding author: Chunyou Sun, sunchy@lzu.edu.cn}
\address{${}^\dagger$ School of Mathematics and Statistics, Lanzhou University, Lanzhou, 730000, P.R. China}
\email{xiaoqingyang22@lzu.edu.cn (X. Yang), andcarva@icmc.usp.br (A. N. Carvalho ), sunchy@lzu.edu.cn (C.Sun)}
\address{${}^\ddag$ ${}^{\mathrm{a}}$ Instituto de Ci{\^e}ncias Matem\'aticas e de Computa\c{c}\~ao, Universidade de S\~ao Paulo Campus de S\~ao Carlos, 
Caixa Postal 668, S\~ao Carlos SP, Brazil}

\address{${}^*$ School of Mathematics and Statistics, Donghua University, Shanghai, 201620, P.R. China}



\begin{document}
	
\begin{abstract}
This paper provides an abstract framework for studying inertial manifolds associated with a class of nonlocal parabolic problems. In particular, by suitably modifying the nonlocal term outside the absorbing ball and changing the scale of time, we derive a corresponding spectral gap condition. As applications, we establish the existence of inertial manifolds for two classes of two-dimensional modified nonlocal parabolic equations on a square domain, whose diffusion coefficients depend on the $L^2$-norm of the solution and of its gradient, respectively.

\end{abstract}
\maketitle	
\tableofcontents
\section{Introduction}
In the study of the asymptotic behavior of an autonomous dissipative partial differential equation (PDE), the global attractor $\mathcal{A}$, which plays a central role, is a compact invariant subset of the infinite-dimensional phase space that attracts all bounded sets under the action of the semigroup $\{\mathcal{S}(t),t\geq0\}$ as time tends to infinity. 
However, due to the complexity of the infinite dimensional dynamical system $\{\mathcal{S}(t),\mathcal{A}\}$,
it is of great interest to identify problems for which the asymptotic behavior of an infinite-dimensional PDE can be captured by finite-dimensional ordinary differential equations (ODEs). The concept of an inertial manifold, introduced by C. Foias, G. R. Sell, and R. Temam in \cite{FST88}, was developed precisely to achieve this goal. Roughly speaking, an inertial manifold is a finite-dimensional invariant Lipschitz submanifold of the phase space that exponentially attracts all orbits.

Consider the following abstract semilinear parabolic problem in a real Hilbert space $\mathcal{H}$ with norm $|\cdot|_{\mathcal{H}}$,
\begin{align}\label{1.11}
	\partial_t u+A^{1+\beta}u=A^{\beta}F(u),\quad u(0)=u_0,
\end{align}
where $\beta\geq0$, $A: D(A)\subset \mathcal{H}\rightarrow \mathcal{H}$ is a positive self-adjoint operator with compact inverse and $F: \mathcal{H} \rightarrow \mathcal{H}$ is globally bounded with bound $K$ and globally Lipschitz continuous with Lipschitz constant $L$. The most direct way to prove the existence of an inertial manifold of \eqref{1.11} is to verify that the corresponding spectral gap condition holds, that is, there exists $N\in\mathbb{Z}^+$ such that
 \begin{align}\label{1.22}
 \frac{\lambda_{N+1}^{1+\beta}-\lambda_{N}^{1+\beta}}{\lambda_{N+1}^{\beta}+\lambda_{N}^{\beta}}>L,
 \end{align}
 where $\{\lambda_n\}_{n=1}^{\infty}$ are the eigenvalues of the operator $A$ arranged in non-decreasing order,
 see \cite{li2024,Anna2015} for more details. It is well known that \eqref{1.22} is naturally satisfied in many one-dimensional (1D) dissipative equations, such as the reaction-diffusion equation, the Kuramoto-Sivashinsky equation,
see \cite{Hale-ODE,Henry,Zelik} and the references therein. However, for higher dimensions, the spectral gap condition depends on the shape of the domain, see, e.g., \cite{Guo, constantin, Temam}. In particular, J. Mallet-Paret and G. R. Sell developed the spatial averaging principle \cite{MS1988} to verify cone invariance and the squeezing property of trajectories, and thus establish the existence of inertial manifolds for a three-dimensional (3D) scalar reaction-diffusion equation posed on a cubic domain. Since then, several research efforts have been made in this direction, such as \cite{Carvalho2022,Guo2018,Guo2024, Anna2023, Anna2015, Li2020,Zelik}.

However, for nonlocal parabolic equations, the relevant theory has not been extensively explored due to the presence of nonlocal terms, which make the analysis more challenging. In particular, the linearized operator is non-autonomous and its spectrum varies along solution trajectories, so that the classical spectral gap condition cannot be directly applied to verify the existence of an inertial manifold. Recently, by employing a time transformation (one for each solution) that converts the nonlocal term into a nonlinear term, and exploiting the Sobolev 
embedding $H_0^1(0,\pi)\hookrightarrow L^{\infty}(0,\pi),$ X. Yang, A. N. Carvalho and E. M. Moreira \cite{Xiaoqing} established a finite-dimensional reduction for a 1D Kirchhoff-type parabolic problem. In particular, for initial data in $H_0^1(0, \pi)$ they constructed an inertial manifold for a suitable modified problem and showed that the global attractor of the original problem admits a finite-dimensional Lipschitz graph representation. These results naturally raise the question of whether a similar reduction persists in higher dimensions and in a more general abstract framework, and, in particular, what form the corresponding spectral gap condition should take.

Motivated by \cite{li2024,Anna2015,Xiaoqing}, we introduce the following abstract model:
\begin{align}\label{1.1}
\partial_t u+a\bigl(|A^{-\frac{\beta}{2}}u|^2_{\mathcal{H}}\bigr)
A^{1+\beta}u=A^{\beta}F(u),\quad u(0)=u_0,
\end{align}
where $\beta$, $A$ and $F$ are as in \eqref{1.11}, and 
$a:\mathbb{R}^+\to [m,M]\subset (0,+\infty)$ is continuously differentiable and globally Lipschitz with Lipschitz constant $\ell_a$. Compared with the semilinear problem \eqref{1.11}, the constant coefficient of $A^{1+\beta}u$ is replaced by $a\bigl(|A^{-\frac{\beta}{2}}u|^2_{\mathcal{H}}\bigr)$, which depends on the global norm of the solution. Hence, problem \eqref{1.1} can be regarded as a nonlocal counterpart of \eqref{1.11}. If the function $a(s)\equiv 1$, model \eqref{1.1} reduces exactly to \eqref{1.11}. 
In particular, choosing $\beta=1$ and $A=-\Delta$ under periodic or Neumann boundary conditions yields the classical Cahn-Hilliard equation studied in \cite{Anna2015}. On the other hand, when $a(\cdot)$ is nonconstant, \eqref{1.1} describes a genuinely nonlocal dynamics. For example, taking $A=-\Delta$ with homogeneous Dirichlet boundary conditions and $\beta=0$ leads to a nonlocal reaction-diffusion equation. In the one-dimensional case $\Omega=(0,\pi)$ with $\mathcal{H}=H_0^1(0,\pi)$, this equation coincides with the Kirchhoff-type model considered in \cite{Xiaoqing}. Consequently, model~\eqref{1.1} encompasses the equations studied in \cite{li2024,Anna2015,Xiaoqing} as special cases, providing a unified framework for the analysis of inertial manifolds in both local and nonlocal settings.

The present paper continues our previous investigation in \cite{Xiaoqing} and aims to formulate the problem within a unified abstract framework. More precisely, we first consider the existence of inertial manifolds for problem \eqref{1.1} within the framework of the invariant cone method (see, e.g., \cite{MS1988,Robinson,Zelik}). Since the corresponding strong cone condition is not straightforward to verify directly, following the approach in \cite{Xiaoqing}, we suitably modify the nonlocal term outside an absorbing ball and change the scale of time. The modified problem has the same asymptotic dynamics as \eqref{1.1} and is transformed into a semilinear parabolic equation of the form \eqref{1.11}. As a consequence, we obtain the following explicit spectral gap condition:
\begin{align}\label{1.2}
\frac{\lambda_{N+1}^{1+\beta}-\lambda_{N}^{1+\beta}}{\lambda_{N+1}^{\beta}+\lambda_{N}^{\beta}}>\frac{L}{m}+\frac{12K\ell_a\rho\lambda_1^{-\beta/2}}{m^2}
\end{align}
which quantitatively reflects the dependence on both the nonlocal coefficient and the nonlinear term, and ensures the existence of an inertial manifold for the suitably modified version of \eqref{1.1}, namely the cut-off problem \eqref{cut-off problem}, to be presented later. Here, $\rho>0$ denotes the radius of the absorbing ball in $\mathcal{H}^{-\beta}$ for problem \eqref{1.1}. 
In addition, we note that if the function $a(s) \equiv1$, then \eqref{1.2} coincides with
\eqref{1.22}. In this sense, \eqref{1.2} can be viewed as a suitable extension of \eqref{1.22}. 

This paper is organized as follows. In Section 2.1,
we introduce some basic definitions and notation, provide the well-posedness, the smoothing property as well as the existence of a compact absorbing ball for problem \eqref{1.1}.
In Section 2.2, by using the invariant cone method, we establish sufficient conditions for the construction of an inertial manifold for the abstract model \eqref{1.1}. In Section 2.3, following the approach in \cite{Xiaoqing}, we suitably modify the nonlocal term and change the scale of time to transform the modified problem into a semilinear parabolic equation. We then derive the spectral gap condition \eqref{1.2}, which guarantees the existence of an inertial manifold for the suitably modified version of \eqref{1.1}.
In Section 3, as applications of the abstract theory established in the previous section, we use a classical result from number theory, introduced by I. Richards (see Lemma \ref{two square}), to verify the corresponding spectral gap conditions for the suitably modified two-dimensional nonlocal parabolic equations whose diffusion coefficients depend on the $L^2$-norm of the solution and the $L^2$-norm of its gradient, respectively, and thus establish the existence of inertial manifolds. In the latter case, following the cut-off technique in \cite{ANA2018}, we employ an additional cut-off construction to overcome the failure of the embedding $H_0^1(\Omega)\hookrightarrow L^\infty(\Omega)$ in two dimensions.

Throughout this article, for convenience, we use $C(\cdot,\cdot)$ to denote a positive constant that depends on its elements and may change from line to line.
\section{Abstract Nonlocal Parabolic Problems}

In this section, we introduce the fundamental definitions and notation, and present the well-posedness and some dissipative estimates for problem \eqref{1.1}. We then establish an abstract criterion for the existence of inertial manifolds based on the invariant cone method. Finally, by suitably modifying the nonlocal term and changing the scale of time, we derive the spectral gap condition \eqref{1.2}, which provides a verifiable condition for the existence of an inertial manifold for the suitably modified problem.

\subsection {Well-Posedness and the Global Attractor}
\noindent

Let $\mathcal{H}$ be a Hilbert space with the norm $|\cdot|_{\mathcal{H}}$,
$A: D(A)\subset \mathcal{H} \rightarrow \mathcal{H}$ is a linear positive self-adjoint operator with compact inverse.
By the Hilbert-Schmidt theorem, there exists a complete orthonormal basis $\{e_n\}_{n=1}^{\infty}$ of $\mathcal{H}$ consisting of eigenvectors of $A$, associated with positive eigenvalues $\lambda_n > 0$ ordered such that 
\begin{align*}
A e_n=\lambda_n e_n, \quad 0<\lambda_1 \leq \lambda_2 \leq \lambda_3 \leq \cdots.
\end{align*}
Due to the compactness of the inverse operator $A^{-1},$ the eigenvalues satisfy $\lambda_n \rightarrow \infty$ as $n \rightarrow \infty$. In addition, every element $u \in \mathcal{H}$ admits a Fourier series representation:

\begin{align}\label{fourier series}
u=\sum_{n=1}^{\infty} u_n e_n,\quad u_n:=\left(u, e_n\right), \quad|u|_{\mathcal{H}}^2=\sum_{n=1}^{\infty} u_n^2,
\end{align}
where ($u, v$) is the inner product in the space $\mathcal{H}$.
As usual, the Hilbert spaces $\mathcal{H}^s$, $s \in \mathbb{R}^{+}$, are defined as follows:
\begin{align}
\mathcal{H}^s:= D(A^{\frac{s}{2}})=\{u \in \mathcal{H};|u|_{\mathcal{H}^s}^2=\sum_{n=1}^{\infty} \lambda_n^s u_n^2<\infty\} .
\end{align}
For $s < 0$, $\mathcal{H}^s$ denotes the completion of $\mathcal{H}$ with respect to the norm $|\cdot|_{\mathcal{H}^s}$. Next, we define the following orthogonal projectors: 
$$
P_N u:=\sum_{n=1}^N u_n e_n,\quad Q_N:=I-P_N,
$$
and define the subspaces $\mathcal{H}_{+}^s:=P_{N}\mathcal{H}^s$ and $\mathcal{H}_{-}^s:=Q_{N}\mathcal{H}^s$ for $N\in\mathbb{Z}^+,s\in\mathbb{R}$.

We consider the following abstract nonlocal parabolic problem:
\begin{equation}\label{nonlocal problem}
	\left\{\begin{aligned}
		&\partial_t u+a(|u|^2_{\mathcal{H}^{-\beta}})A^{1+\beta}u=A^{\beta}F(u),\ t>0,\\
		&u(0)=u_0,
	\end{aligned}\right.
\end{equation} 
where $\beta\geq0$, $u_0 \in \mathcal{H}^{-\beta}$ is a given initial datum, $a:\mathbb{R^+}\to [m,M]\subset (0,+\infty)$ is a continuously differentiable function and is globally Lipschitz with Lipschitz constant $\ell_a$. The nonlinearity $F: \mathcal{H}\rightarrow\mathcal{H}$ is globally Lipschitz with Lipschitz constant $L$ and globally bounded, i.e.,
\begin{equation}\label{F}
	\begin{aligned}
	|F(u)| _\mathcal{H}&\leq K, \quad u \in \mathcal{H},\\
	|F\left(u_1\right)-F\left(u_2\right)|_\mathcal{H}&\leq L|u_1-u_2|_\mathcal{H}, \quad u_1, u_2 \in \mathcal{H}.
	\end{aligned}
\end{equation}

We say that a function $u(t)$ is a solution to \eqref{nonlocal problem} on an interval $[0, T]$ if $u \in C([0, T] ; \mathcal{H}^{-\beta}) \cap L^2(0, T ; \mathcal{H}^{1})$ 
and \eqref{nonlocal problem} is satisfied in the sense of distributions, i.e., $u(0)=u_0$ and
$$
-\int_0^T\left(u, \phi_t\right) dt+\int_0^T(a(|u|_{\mathcal{H}^{-\beta}}^2)A^{1+\beta}u, \phi) dt=\int_0^T(A^{\beta}F(u), \phi) dt,
$$
for every $\phi \in C^\infty_{c}([0, T];\mathcal{H}^{1+2\beta})$.

Now, we using a standard way to establish the well-posedness of problem \eqref{nonlocal problem}.
\begin{proposition}\label{well poseness}
 Let the nonlinearity $F$, the operator $A$ and the function $a(s)$ satisfy the above assumptions. If $u_0 \in \mathcal{H}^{-\beta}$ with $\beta\geq0,$ then problem \eqref{nonlocal problem} admits a unique solution $u \in C([0, T] ; \mathcal{H}^{-\beta}) \cap L^2(0, T ; \mathcal{H}^{1})$ for all $T>0$. Furthermore, \eqref{nonlocal problem} holds as an equality in $L^2(0, T ; \mathcal{H}^{-1-2\beta})$.
\end{proposition}
\begin{proof}
	We employ the Galerkin method to prove the well-posedness of the nonlocal problem \eqref{nonlocal problem}. For each integer $n\geq1$, define $V_n:=\operatorname{span}\left\{e_k: 1\leq k\leq n\right\}$.
	Let us consider the Galerkin approximate solution to \eqref{nonlocal problem} as follows:
	$$
	u_n(t,x)=\sum_{1\leq k\leq n} {g}_k(t)e_k(x),
	$$
	which solves the following system of ODEs: 
	\begin{equation}
		\left\{\begin{aligned}
			&\frac{d}{dt}(u_n,{e}_{k})+a(|u_n|^2_{\mathcal{H}^{-\beta}})(A^{1+\beta}u_n, {e}_{k})=(A^{\beta}F(u_n),{e}_{k}),\\
			&(u_n(0),{e}_{k})=(u_0,{e}_{k}),\quad\quad 1\leq k\leq n.
		\end{aligned}\right.
	\end{equation} 
	It follows that
	\begin{equation}
		\left\{\begin{aligned}
			&\frac{d}{dt}g_{k}(t)+a(|u_n|^2_{\mathcal{H}^{-\beta}})\lambda_{k}^{1+\beta}g_{k}(t)=(A^{\beta}F(u_n),{e}_{k}),\\
			&g_{k}(0)=(u_0,{e}_{k}),\quad\quad 1\leq k\leq n,
		\end{aligned}\right.
	\end{equation} 
	which can be rewritten as	\begin{equation}\label{Approximate system}
		\left\{\begin{aligned}
			&\frac{d}{dt}u_n+a(|u_n|^2_{\mathcal{H}^{-\beta}})A^{1+\beta}u_n=P_nA^{\beta}F(u_n),\\
			&u_n(0)=P_nu_0.
		\end{aligned}\right.
	\end{equation} 
According to the existence and uniqueness theorem for ODEs, one deduces that \eqref{Approximate system} admits a unique solution $u_n \in C\left(\left[0, \infty\right) ; \mathcal{H}^{-\beta}\right) \cap C^1\left(\left(0, \infty\right) ; \mathcal{H}^{-\beta}\right)$.
	
	Multiplying \eqref{Approximate system} with $A^{-\beta}u_n$ and using \eqref{F}, we get
	\begin{align*}
		\frac{1}{2} \frac{d}{dt}|u_n|^2_{\mathcal{H}^{-\beta}}+a(|u_n|^2_{\mathcal{H}^{-\beta}})|A^{\frac12}u_n|^2_{\mathcal{H}}=(F(u_n),u_n)\leq|F(u_n)|_{\mathcal{H}}|u_n|_{\mathcal{H}}\leq K|u_n|_{\mathcal{H}}.
	\end{align*}
	Moreover, by Young's inequality, we have
	\begin{align}\label{3.2}
	 \frac{d}{dt}|u_n|^2_{\mathcal{H}^{-\beta}}+m|u_n|^2_{\mathcal{H}^{1}}\leq\frac{K^2}{m\lambda_1}.
	\end{align}
Integrating the above inequality from $0$ to $T$, we have
	\begin{align}\label{3.4}
		|u_n(T)|^2_{\mathcal{H}^{-\beta}}+	\int_{0}^{T}m|u_n(s)|^2_{\mathcal{H}^{1}}ds\leq|u_n(0)|^2_{\mathcal{H}^{-\beta}}+\frac{K^2}{m\lambda_1}T.
	\end{align}
	Therefore, the sequence $\{u_n\}$ is uniformly bounded in $ L^{\infty}(0,T;\mathcal{H}^{-\beta})\cap L^2(0,T;\mathcal{H}^{1})$. 
	It follows from \eqref{F} that the sequence $\{F(u_n)\}$ is uniformly bounded in $L^{2}(0,T;\mathcal{H})$, for all $T>0$. 
	As a result, $\{\frac{d}{dt}u_n\}$ is uniformly bounded in $L^{2}(0,T;\mathcal{H}^{-1-2\beta})$, for all $T>0$.
	Consequently, there exist a subsequence $\{u_n\}$ (still denoted by $\{u_n\}$) and $u \in L^{\infty}(0,T; \mathcal{H}^{-\beta})$ such that
	$$
	\begin{aligned}
		&u_n  \stackrel{*}{\rightharpoonup}u \quad\text {in} \ L^{\infty}(0, T ; \mathcal{H}^{-\beta}), \\
		&u_n {\rightharpoonup}u\quad\text {in}~L^2(0, T ; \mathcal{H}^{1}), \\
		&F\left(u_n\right){\rightharpoonup} \chi \quad\text {in}~L^{2}\left(0, T ; \mathcal{H}\right), \\
		&\frac{d}{dt}u_n{\rightharpoonup}\frac{d}{dt}u \quad \text {in}~ L^2(0, T ; \mathcal{H}^{-1-2\beta}),\\
		&a(|u_n|_{\mathcal{H}^{-\beta}}^2) \stackrel{*}{\rightharpoonup}b\quad \text { in}~L^{\infty}(0, T),
	\end{aligned}
	$$
	where $\rightharpoonup(\stackrel{*}{\rightharpoonup})$ stands for the weak (weak star) convergence. 
	 According to the Aubin-Lions compactness theorem \cite{Chueshov}, we infer that 
	\begin{align*}
	u_n \rightarrow u \quad\text{in}\ L^2(0,T; \mathcal{H}).
	\end{align*}
	It follows from \eqref{F} that $\chi=F(u)$. Moreover, by the continuity of the function $a(s)$, we obtain that $b=a(|u|^2_{\mathcal{H}^{-\beta}})$ and
	\begin{align*}
	a(|u_n|^2_{\mathcal{H}^{-\beta}})\rightarrow a(|u|^2_{\mathcal{H}^{-\beta}})\quad \text{in} \ L^2(0, T).
	\end{align*}
	Thus, we have shown that all terms in the first equation of \eqref{Approximate system} converge weakly in $L^2(0, T ; \mathcal{H}^{-1-2\beta}).$ In addition, we see that $u\in C([0,T];\mathcal{H}^{-\beta})\cap L^2(0, T ; \mathcal{H}^{1})$.
	
	Next, we intend to show that $u(0)=u_0$. In fact, choosing a test function $\phi \in C^1([0, T];\mathcal{H}^{1+2\beta})$ with $\phi(T)=0$, then we get 
	\begin{align}\label{2.11}
		\langle \partial_tu,\phi\rangle+\langle a(|u|^2_{{\mathcal{H}}^{-\beta}})A^{1+\beta}u ,\phi\rangle=\langle A^{\beta}F(u),\phi\rangle,
	\end{align}
	where $\langle\cdot,\cdot\rangle$ denotes the scalar product for the duality $\mathcal{H}^{-1-2\beta},\mathcal{H}^{1+2\beta}$.
Integrating \eqref{2.11} from 0 to $T$ and using integration by parts for the first term, we have
	\begin{align*}
	\int_0^T-\langle u, \phi_{t}\rangle+\langle a(|u|^2_{{\mathcal{H}}^{-\beta}})A^{1+\beta}u ,\phi\rangle d s=\int_0^T\langle A^{\beta}F(u),\phi\rangle d s+\langle u(0), \phi(0)\rangle.
	\end{align*}
If we perform the same procedure on the Galerkin approximate equation \eqref{Approximate system}, we obtain
	\begin{align*}
	\int_0^T-\langle u_n, \phi_{t}\rangle+\langle a(|u_n|^2_{{\mathcal{H}}^{-\beta}})A^{1+\beta}u_n ,\phi\rangle d s=\int_0^T\langle P_nA^{\beta}F(u_n),\phi\rangle d s+\langle u_n(0), \phi(0)\rangle.
\end{align*}
Therefore, we can take the limits of all these terms in the above equality. Since $u_n(0)=P_n(u_0) \rightarrow u_0$ in $\mathcal{H}^{-\beta}$ as $n\rightarrow\infty$ and $\phi(0)$ is arbitrary, it follows that $u(0)=u_0$, as required.

It remains to show that the solution to \eqref{nonlocal problem} is unique. Assume that $u$ and $v$ are two solutions of \eqref{nonlocal problem} with the same initial condition $u_0$, set $w=u-v$, which satisfies 
\begin{equation}\label{Difference of solution}
	\left\{\begin{aligned}
		& w_t+a(|u|^2_{{\mathcal{H}}^{-\beta}})A^{1+\beta}u-a(|v|^2_{\mathcal{H}^{-\beta}})A^{1+\beta} v=A^{\beta}(F(u)-F(v)),\\
		& w(0)=0.
	\end{aligned}\right. 
\end{equation}
Multiplying \eqref{Difference of solution} with $A^{-\beta}w$, we get
\begin{align*}
	&\frac{1}{2} \frac{d}{d t}|w|^2_{\mathcal{H}^{-\beta}}+a\left(|u(t)|^2_{{\mathcal{H}}^{-\beta}}\right)(Aw,w)\\
	&=(F(u)-F(v), w)+\left(a\left(|v(t)|^2_{{\mathcal{H}}^{-\beta}}\right)-a\left(|u(t)|^2_{{\mathcal{H}}^{-\beta}}\right)\right)(Av,w)\\
	&\leq L|w|_{\mathcal{H}}^2+\ell_a(|u|_{{\mathcal{H}}^{-\beta}}+|v|_{{\mathcal{H}}^{-\beta}})|v|_{\mathcal{H}^{1}}|w|_{{\mathcal{H}}^{1}}|w|_{{\mathcal{H}}^{-\beta}}\\
	&\leq\frac{m}{4}|w|_{{\mathcal{H}}^{1}}^2+C(m,L)|w|_{{\mathcal{H}}^{-\beta}}^2+\frac{m}{4}|w|_{{\mathcal{H}}^{1}}^2+\frac{2}{m}\ell_a^2(|u|_{{\mathcal{H}}^{-\beta}}^2+|v|_{{\mathcal{H}}^{-\beta}}^2)|v|_{\mathcal{H}^{1}}^2|w|_{{\mathcal{H}}^{-\beta}}^2,
\end{align*}
where we have used the following interpolation inequality:
\begin{align*}
  |w|_{\mathcal{H}}^2\leq|w|_{{\mathcal{H}}^{1}}^{\frac{2\beta}{1+\beta}} |w|_{{\mathcal{H}}^{-\beta}}^{\frac{2}{1+\beta}}.
  \end{align*}
Therefore, we have 
\begin{align}\label{lipschitz continuous}
	\frac{d}{d t}|w|^2_{\mathcal{H}^{-\beta}}+m|w|^2_{\mathcal{H}^{1}}
	&\leq(C(m,L)+\frac{4}{m}\ell_a^2(|u|_{{\mathcal{H}}^{-\beta}}^2+|v|_{{\mathcal{H}}^{-\beta}}^2)|v|_{\mathcal{H}^{1}}^2)|w|_{{\mathcal{H}}^{-\beta}}^2,
\end{align}
by Gronwall's inequality and \eqref{3.4}, for any $T>0,$ we obtain 
\begin{align}\label{continuous}
	|w(t)|^2_{{\mathcal{H}}^{-\beta}} \leq 
	e^{C(m,\ell_a,\lambda_1,L,K,T,|u_0|_{{\mathcal{H}}^{-\beta}})t}|w(0)|^2_{{\mathcal{H}^{-\beta}}},\quad t\in[0,T].
\end{align}
Therefore, we obtain the desired result.
\end{proof}
\begin{remark}\label{remark2.2}
{
It should be noted that the solution $u(t,u_0)$ of \eqref{nonlocal problem} admits the following variation of constants formula:
\begin{align}\label{formula1}
u(t,u_0)=U_{u_0}(t)u_0+\int_0^tU_{u(\theta,u_0)}(t-\theta)A^\beta F(u(\theta,u_0))\,d\theta,
\end{align}
where 
$$
U_{u_0}(t)=e^{-A^{1+\beta}\int_0^ta(|u(r,u_0)|_{\mathcal{H}^{-\beta}}^2)dr},\quad 
U_{u(\theta)}(t-\theta)=e^{-A^{1+\beta}\int_0^{t-\theta}a(|u(r, u(\theta))|_{\mathcal{H}^{-\beta}}^2)dr}.
$$ 
See, e.g., \cite{Yagi,Xiaoqing} for related arguments.
}
\end{remark}

Let $u(t)$ denote the solution of \eqref{nonlocal problem}, the associated semigroup $S_a(t):\mathcal{H}^{-\beta}\to\mathcal{H}^{-\beta}$ is defined by $S_a(t)u_0:=u(t)$ for $t\geq0.$ 
Moreover, we see that this semigroup possesses a compact absorbing ball in the phase space $\mathcal{H}^{-\beta}$, that is, the following result holds.

\begin{proposition}\label{global attractor}
Suppose that the assumptions of Proposition \ref{well poseness} are satisfied. Then there exists a constant $\rho_{\beta}>0$ such that the semigroup $S_a(t)$ associated with \eqref{nonlocal problem} admits an absorbing ball $\mathcal{B}_{\rho_\beta}:=\{u\in \mathcal{H}^{1-\beta}; |u|_{\mathcal{H}^{1-\beta}}\leq\rho_{\beta}\}$
in the phase space $\mathcal{H}^{-\beta}$.
Moreover, $S_a(t)$ has a global attractor $\mathcal{A}\subset {\mathcal{H}^{1-\beta}}$.
\end{proposition}
\begin{proof}
	Here we only give a formal derivation, which can be justified by the Galerkin method.
 Taking the inner scalar product of \eqref{nonlocal problem} with $A^{-\beta}u$, we get
	\begin{align*}
		\frac{d}{dt}|u|^2_{\mathcal{H}^{-\beta}}+2a(|u|^2_{\mathcal{H}^{-\beta}})|u|^2_{\mathcal{H}^{1}}
		\leq 2K|u|_{\mathcal{H}}\leq2K\lambda_1^{-\frac12}|u|_{\mathcal{H}^1}\leq m|u|_{\mathcal{H}^1}^2+\frac{K^2}{m\lambda_1},
	\end{align*}
	which implies that
		\begin{align}\label{2.13}
		\frac{d}{dt}|u|^2_{\mathcal{H}^{-\beta}}+m|u|^2_{\mathcal{H}^{1}}
		\leq \frac{K^2}{m\lambda_1}.
	\end{align}
	Hence,
	\begin{align*}
		\frac{d}{dt}|u|^2_{\mathcal{H}^{-\beta}}+m\lambda_1^{1+\beta}|u|^2_{\mathcal{H}^{-\beta}}
		\leq \frac{K^2}{m\lambda_1}.
	\end{align*}
Applying Gronwall's inequality, we obtain
	\begin{align}\label{solution estimate}
		|u(t)|^2_{\mathcal{H}^{-\beta}}
	\leq e^{-m\lambda_1^{1+\beta}t} |u_0|^2_{\mathcal{H}^{-\beta}}+\frac{K^2}{m^2\lambda_1^{2+\beta}}(1-e^{-m\lambda_1^{1+\beta}t}).
	\end{align}
Thus, there exists $t_0>0$ such that
\begin{align*}
|u(t)|^2_{\mathcal{H}^{-\beta}}\leq2\frac{K^2}{m^2\lambda_1^{2+\beta}}=:\rho_0^2, \quad\text{for}\ t\geq t_0.
\end{align*}
Furthermore, integrating \eqref{2.13} from $t$ to $t+1$, we infer that 
	\begin{align}\label{2.14}
		\int_t^{t+1}|u(\tau)|^2_{\mathcal{H}^{1}}d\tau&\leq \frac{1}{m}|u(t)|^2_{\mathcal{H}^{-\beta}}+\frac{K^2}{m^2\lambda_1}\\ \nonumber
		&\leq\frac{1}{m}\rho_0^2+\frac{K^2}{m^2\lambda_1}   
   =:\rho_1^2, \quad\text{for}\ t\geq t_0.
	\end{align}
Taking the inner scalar product of \eqref{nonlocal problem} with $A^{1-\beta}u$, we get
	\begin{align*}
	\frac{d}{dt}|u|^2_{\mathcal{H}^{1-\beta}}+2m|u|^2_{\mathcal{H}^{2}}
	\leq2K|u|_{\mathcal{H}^{2}}\leq m|u|^2_{\mathcal{H}^{2}}+\frac{K^2}{m}.
	\end{align*}
This means that
	\begin{align}\label{2.166}
	\frac{d}{dt}|u|^2_{\mathcal{H}^{1-\beta}}+m|u|^2_{\mathcal{H}^{2}}
	\leq \frac{K^2}{m}.
\end{align}
Applying the uniform Gronwall inequality to \eqref{2.166}, it follows from \eqref{2.14} that
\begin{align*}
	|u(t)|^2_{\mathcal{H}^{1-\beta}}\leq\rho_1^2+\frac{K^2}{m}=:\rho_2^2, \quad\text{for}\ t\geq t_0+1.
\end{align*}
Furthermore, integrating \eqref{2.166} from $t$ to $t+1$, we infer that 
	\begin{align}
	\int_t^{t+1}|u(\tau)|^2_{\mathcal{H}^{2}}d\tau\leq \frac{1}{m}|u(t)|^2_{\mathcal{H}^{1-\beta}}+\frac{K^2}{m^2}\leq\frac{1}{m}\rho_2^2+\frac{K^2}{m^2}=:\rho_3^2, \quad\text{for}\ t\geq t_0+1.
\end{align}
This completes the proof.
\end{proof}

As we will see soon, the solution $u(t)$ to \eqref{nonlocal problem}
exhibits an instantaneous smoothing effect from $\mathcal{H}^{-\beta}$ to $\mathcal{H}^{\gamma}$ with $-\beta\leq\gamma<2$. To be precise, we have the following proposition.

\begin{proposition}\label{prop2.5}
Let the assumptions of Proposition \ref{well poseness} be satisfied. Then, for any $u_0\in\mathcal{H}^{-\beta}$, we have
$u(t)\in\mathcal{H}^{\gamma}$ for all $t>0$ and for every $\gamma\in[-\beta, 2)$. Moreover, 
there exist constants $C_{\gamma}>0$ such that
\begin{align}\label{smoothness1}
|u(t)|_{\mathcal{H}^{\gamma}}\leq C(m)t^{-\frac{(\gamma+\beta)}{2(1+\beta)}}|u_0|_{\mathcal{H}^{-\beta}}+ C_{\gamma}, \quad -\beta\leq\gamma<2,\ t\in(0,1].
\end{align}
\end{proposition}
\begin{proof}
Indeed, we observe that there exists a constant $C(m)>0$ such that
\begin{align}\label{22.25}
|U_{u_0}(t)|_{\mathcal{L}(\mathcal{H}^{-s},\mathcal{H}^{\gamma})}\leq C(m)t^{-\frac{(\gamma+s)}{2(1+\beta)}},\quad\text{for any } s,\gamma\geq0 \text{ and } t>0.
\end{align}
For any $-\beta\leq\gamma<2$, from \eqref{formula1} and \eqref{22.25}, we deduce that
\begin{align}\label{semigroup es}
|u(t)|_{\mathcal{H}^{\gamma}}&\leq|U_{u_0}(t)u_0|_{\mathcal{H}^{\gamma}}+\int_{0}^{t}|U_{u(\theta)}(t-\theta)A^{\beta}F(u(\theta))|_{\mathcal{H}^{\gamma}}d\theta \nonumber\\
&\leq C(m)t^{-\frac{(\gamma+\beta)}{2(1+\beta)}}|u_0|_{\mathcal{H}^{-\beta}}+\int_0^t|U_{u(\theta)}(t-\theta)|_{\mathcal{L}(\mathcal{H}^{-2\beta},\mathcal{H}^{\gamma})}|A^{\beta}|_{\mathcal{L}(\mathcal{H},\mathcal{H}^{-2\beta})}|F(u(\theta))|_{\mathcal{H}}d\theta \nonumber\\
&\leq C(m)t^{-\frac{(\gamma+\beta)}{2(1+\beta)}}|u_0|_{\mathcal{H}^{-\beta}}
+C(m,K,|A^{\beta}|_{\mathcal{L}(\mathcal{H},\mathcal{H}^{-2\beta})})\int_0^t
(t-\theta)^{-\frac{(\gamma+2\beta)}{2(1+\beta)}}d\theta\\
&\leq C(m)t^{-\frac{(\gamma+\beta)}{2(1+\beta)}}|u_0|_{\mathcal{H}^{-\beta}}
+C(m,\gamma, \beta, K,|A^{\beta}|_{\mathcal{L}(\mathcal{H},\mathcal{H}^{-2\beta})})t^{\frac{2-\gamma}{2(1+\beta)}} \nonumber\\
&\leq C(m)t^{-\frac{(\gamma+\beta)}{2(1+\beta)}}|u_0|_{\mathcal{H}^{-\beta}}+C_{\gamma}, \quad\text{for }t\in(0,1].\nonumber
\end{align}
\end{proof}

By virtue of Propositions \ref{global attractor} and \ref{prop2.5}, we obtain the following dissipative estimates.
\begin{corollary}\label{corollary2.6}
Let the assumptions of Proposition \ref{prop2.5} be satisfied. Then, there exist constants $C_1,C_2,R_1>0,$ independent of $N,$ such that
\begin{align}\label{Compact dissipative}
|u(t)|_{\mathcal{H}}\leq C_1e^{-\frac m2\lambda_1^{1+\beta}t}|u_0|_{\mathcal{H}^{-\beta}}+C_2,\quad\text{for }\ t\geq1.
\end{align}
Moreover, if $u_0\in\mathcal{H},$ for every $N\in \mathbb{N}$, we have
\begin{align}\label{dissipative in H}
|Q_Nu(t)|_{\mathcal{H}}\leq C_1e^{-\frac m2\lambda_{1}^{1+\beta}t}|Q_Nu_0|_{\mathcal{H}}+R_1,\quad \text{for } t>0.
\end{align}
\end{corollary}
\begin{proof}
It follows from \eqref{solution estimate} and \eqref{smoothness1} that
\begin{align*}
|u(t)|_{\mathcal{H}}&\leq C(m)|u(t-1)|_{\mathcal{H}^{-\beta}}+C_0\\
&\leq C(m)\left(e^{-\frac{m}{2}\lambda_1^{1+\beta}(t-1)} |u_0|_{\mathcal{H}^{-\beta}}+C(m,\lambda_1,\beta,K)\right)+C_0\\
&\leq C(m,\lambda_1,\beta)e^{-\frac m2\lambda_1^{1+\beta}t}|u_0|_{\mathcal{H}^{-\beta}}+C(m,\lambda_1,\beta,K,C_0),\quad\text{for }\ t\geq1.
\end{align*}
Hence, \eqref{Compact dissipative} holds.
In view of \eqref{formula1}, we see that
$$
Q_Nu(t)=U_{u_0}(t)Q_Nu_0+\int_{0}^{t}U_{u(\theta)}(t-\theta)
Q_NA^{\beta}F(u(\theta))d\theta.
$$
In a similar manner to \eqref{semigroup es} and \eqref{Compact dissipative}, we obtain that
\begin{align}\label{2.331}
|Q_Nu(t)|_{\mathcal{H}}\leq C_1e^{-\frac m2\lambda_{1}^{1+\beta}t}|Q_Nu_0|_{\mathcal{H}^{-\beta}}+C_2,\quad\text{for }\ t\geq1,
\end{align}
and
\begin{align}\label{short}
|Q_Nu(t)|_{\mathcal{H}}
&\leq C(m)|Q_Nu_0|_{\mathcal{H}}+C_0, \quad\text{for } t\in(0,1].
\end{align}
As a consequence of \eqref{2.331} and \eqref{short}, we arrive at \eqref{dissipative in H}.
\end{proof}

\subsection{The Existence of Inertial Manifolds}
\noindent

In this subsection, we extend the invariant cone method presented in, e.g., \cite{Anna2015,sell2002,Zelik} for the semilinear parabolic case to a class of nonlocal parabolic problems. In addition, we derive an explicit spectral gap condition for the modified nonlocal problem, which guarantees the existence of inertial manifolds.

We now recall the concept of the inertial manifold 
for the semigroup.

\begin{definition}\label{IM}
A subset $\mathcal{M}\subset \mathcal{H}^{-\beta}$
	is called an inertial manifold for the semigroup $S_a(t)$ associated with \eqref{nonlocal problem}, provided the following conditions are
	satisfied:
	\begin{enumerate}
		\item $\mathcal{M}$ is invariant, i.e., $S_a(t)\mathcal{M}=\mathcal{M}$, for all $t\geq 0$.
		\item $\mathcal{M}$ is a finite-dimensional Lipschitz manifold, i.e., there
		exists a Lipschitz continuous function $\Phi :\mathcal{H}^{-\beta}_{+}\rightarrow \mathcal{H}^{-\beta}_{-}$ such
		that $\mathcal M$ is the graph of $\Phi$, that is, 
		\begin{equation*}
			\mathcal{M} =\{u \in\mathcal{H}^{-\beta}; \; u = p+\Phi (p),\,p\in \mathcal{H}^{-\beta}_{+}\}.
		\end{equation*}
		\item The exponential tracking property holds; that is, there exists a constant $\alpha>0$ such that, for every $u_0\in\mathcal{H}^{-\beta}$, there exist $v_0\in\mathcal{M}$ and a constant $C=C(u_0,v_0)>0$ satisfying
      \begin{equation}\label{tracking}
      |S_a(t)u_0-S_a(t)v_0|_{\mathcal{H}^{-\beta}}
      \leq C e^{-\alpha t},\quad \text{for all } t\ge0.
        \end{equation}
       \end{enumerate}
\end{definition}

In what follows, for convenience, we introduce the following notation:
$$
V(w):=|Q_Nw|_{\mathcal{H}^{-\beta}}^2-|P_Nw|_{\mathcal{H}^{-\beta}}^2, \quad w\in{\mathcal{H}^{-\beta}},
$$
$$
K^{+}:=\{w\in{\mathcal{H}^{-\beta}}; V(w)\leq 0\}.
$$
With this notation, we recall the definition of the strong squeezing property.

\begin{definition}
We say that the semigroup $S_a(t)$ associated with \eqref{nonlocal problem} possesses 
the strong squeezing property, if the following properties hold:
\begin{enumerate}
\item Invariance of the cone $K^{+}$: if
\begin{align}\label{invariance of cone}
u_1-u_2 \in K^{+} \Rightarrow S_a(t)u_1-S_a(t)u_2 \in K^{+},\ 
\end{align}
 for all $u_1, u_2 \in \mathcal{H}^{-\beta}$ and 
 $t\geq 0$.
\item The squeezing property: there exist positive constants
$\gamma$ and $C$ such that, for every $T>0$ and every
$u_1,u_2\in\mathcal H^{-\beta}$,
\begin{align}
S_a(T)u_1-S_a(T)u_2\notin K^+
\Rightarrow
|S_a(t)u_1-S_a(t)u_2|_{\mathcal H^{-\beta}}
\leq
Ce^{-\gamma t}|u_1-u_2|_{\mathcal H^{-\beta}},
\end{align}
for all $t\in[0,T]$.
\end{enumerate}
\end{definition}

Now, we state the following result.
\begin{theorem}\label{abstract theorem}
Let the assumptions of Proposition \ref{well poseness} be satisfied. In addition, suppose that the solution semigroup $S_a(t)$ associated with \eqref{nonlocal problem} satisfies the strong squeezing property. Then problem \eqref{nonlocal problem} possesses an $N$-dimensional inertial manifold in the sense of Definition \ref{IM}.
\end{theorem}
\begin{proof}
As shown in \cite{Robinson}, the strong squeezing property implies the existence of an inertial manifold via the graph transform method. Here, we provide an alternative proof of this result by following a similar idea to that in \cite{Zelik}. The argument can be divided into three steps.\\
{\textbf{Step 1.}} We aim to prove that the following boundary value problem:
\begin{equation}\label{backward solution}
	\left\{\begin{aligned}
		&\partial_tu+a(|u|^2_{\mathcal{H}^{-\beta}})A^{1+\beta}u=A^{\beta}F(u),\\
		&P_N u(0)=u_0^+, \quad Q_N u(-T)=0,
	\end{aligned}\right.
\end{equation} 
has a unique solution for any $T>0$ and $u_0^+\in\mathcal{H}_{+}^{-\beta}$. In fact, we consider the map $G_T$ : $\mathcal{H}_{+}^{-\beta}\rightarrow \mathcal{H}_{+}^{-\beta}$ given by
$$
G_T(v^+)=P_N S_a(T)v^+,\quad\text {for}\ v^+ \in \mathcal{H}_{+}^{-\beta}.
$$
By virtue of \eqref{continuous}, we note that $G_T$ is Lipschitz continuous. Next we will show that the map $G_T$ is injective. Assume that $G_T(v_1^+)=G_T(v_2^+)$
for some $v_1^+,v_2^+\in\mathcal{H}_{+}^{-\beta}$, and define
\[
u_i(t):=S_a(t+T)v_i^+,
\quad t\in[-T,0],\quad i=1,2.
\]
Set $w(t)=u_1(t)-u_2(t)$, from the boundary value condition of \eqref{backward solution}, we see that $w(-T)\in K^{+}$. 
It follows from \eqref{invariance of cone} that $w(t)\in K^{+}$ for $t\in[-T,0]$.
Let $P_Nw=p$, $Q_Nw=q$, ${\bar{a}(w)=a(|u_1|^2_{ \mathcal{H}^{-\beta}})-a(|u_2|^2_{ \mathcal{H}^{-\beta}})}$, then one has
\begin{align}\label{p}
	\partial_tp+a(|u_1|^2_{{\mathcal{H}}^{-\beta}})A^{1+\beta}p+\bar{a}(w)A^{1+\beta}P_{N}u_2=A^{\beta}P_{N}({F}(u_1)-{F}(u_2)),
	\end{align}
\begin{align}\label{q}
	\partial_tq+a(|u_1|^2_{{\mathcal{H}}^{-\beta}})A^{1+\beta}q+\bar{a}(w)A^{1+\beta}Q_{N}u_2=A^{\beta}Q_{N}({F}(u_1)-{F}(u_2)),
\end{align}
Multiplying \eqref{p} by $A^{-\beta}p$, we have
\begin{align}\label{2.25}
\frac{1}{2} \frac{\mathrm{d}}{\mathrm{d}t}|p|_{\mathcal{H}^{-\beta}}^2&=-a(|u_1|_{\mathcal{H}^{-\beta}}^2)(Ap,p)-\bar{a}(w)(AP_Nu_2,p)+(F(u_1)-F(u_2),p)\\ \nonumber
&\geq-M\lambda_{N}^{1+\beta}|p|_{\mathcal{H}^{-\beta}}^2-\ell_a(|u_1|_{\mathcal{H}^{-\beta}}+|u_2|_{\mathcal{H}^{-\beta}})|AP_Nu_2|_{\mathcal{H}}|w|_{\mathcal{H}^{-\beta}}|p|_{\mathcal{H}}\\
&\quad-L(|p|_{\mathcal{H}}^2+|p|_{\mathcal{H}}|q|_{\mathcal{H}}).
\end{align}
Moreover, we observe that
\begin{align}\label{2.26}
	L(|p|_{\mathcal{H}}^2+|p|_{\mathcal{H}}|q|_{\mathcal{H}})&\leq L\lambda_{N}^{\beta}|p|_{\mathcal{H}^{-\beta}}^2+L\lambda_{N}^{\frac{\beta}{2}}\lambda_{N+1}^{-\frac12}|p|_{\mathcal{H}^{-\beta}}|q|_{\mathcal{H}^{1}}\\\nonumber
	&\leq L\lambda_{N}^{\beta}(1+\frac{1}{m}L\lambda_{N+1}^{-1})|p|_{\mathcal{H}^{-\beta}}^2+\frac{m}{4}|q|_{\mathcal{H}^{1}}^2,
\end{align}
and 
\begin{align}\label{2.27}
&(|u_1|_{\mathcal{H}^{-\beta}}+|u_2|_{\mathcal{H}^{-\beta}})|AP_Nu_2|_{\mathcal{H}}|w|_{\mathcal{H}^{-\beta}}|p|_{\mathcal{H}} \nonumber\\
&\leq(|u_1|_{\mathcal{H}^{-\beta}}+|u_2|_{\mathcal{H}^{-\beta}})\lambda_{N}^{1+\frac{\beta}{2}}|u_2|_{\mathcal{H}^{-\beta}}\lambda_{N}^{\frac{\beta}{2}}|p|_{\mathcal{H}^{-\beta}}(|p|^2_{\mathcal{H}^{-\beta}}+|q|^2_{\mathcal{H}^{-\beta}})^{\frac12}\\ \nonumber
&\leq\sqrt{2}\lambda_{N}^{1+{\beta}}(|u_1|_{\mathcal{H}^{-\beta}}+|u_2|_{\mathcal{H}^{-\beta}})|u_2|_{\mathcal{H}^{-\beta}}|p|_{\mathcal{H}^{-\beta}}^2\\
&\leq\frac{\sqrt{2}}{2}\lambda_{N}^{1+{\beta}}(|u_1|_{\mathcal{H}^{-\beta}}^2+3|u_2|_{\mathcal{H}^{-\beta}}^2)|p|_{\mathcal{H}^{-\beta}}^2. \nonumber
\end{align}
Taking into account \eqref{2.25}, \eqref{2.26} and \eqref{2.27}, we deduce that
\begin{align*}
&\frac{\mathrm{d}}{\mathrm{d}t}|p|_{\mathcal{H}^{-\beta}}^2\\
&\geq-2\lambda_{N}^{1+\beta}\left(M+\frac{\sqrt{2}}{2}\ell_a(|u_1|_{\mathcal{H}^{-\beta}}^2+3|u_2|_{\mathcal{H}^{-\beta}}^2)+L\lambda_{N}^{-1}(1+\frac{1}{m}L\lambda_{N+1}^{-1})\right)|p|_{\mathcal{H}^{-\beta}}^2+\frac{m}{2}|q|_{\mathcal{H}^{1}}^2.
\end{align*}
Integrating the above inequality from $t<0$ to $0$, we obtain that
\begin{align}\label{2.28}
|p(t)|_{\mathcal{H}^{-\beta}}^2&\leq|p(0)|_{\mathcal{H}^{-\beta}}^2+\int_{t}^{0}\frac{m}{2}|q(s)|^2_{\mathcal{H}^{1}}ds\\ \nonumber
&\quad+\int_{t}^{0}\lambda_{N}^{1+\beta}\left(2M+\sqrt{2}\ell_a(|u_1|_{\mathcal{H}^{-\beta}}^2+3|u_2|_{\mathcal{H}^{-\beta}}^2)+2L\lambda_{N}^{-1}(1+\frac{1}{m}L\lambda_{N+1}^{-1})\right)|p|_{\mathcal{H}^{-\beta}}^2ds.
\end{align}
Furthermore, we have
\begin{align}\label{2.29}
&\int_{t}^{0}\frac{m}{2}|q(s)|_{\mathcal{H}^{1}}^2ds\\ \nonumber
&\leq\frac{1}{2}|p(t)|_{\mathcal{H}^{-\beta}}^2+\int_{t}^{0}\left(\frac{2\ell_a^2}{m}(|u_1|_{\mathcal{H}^{-\beta}}^2+|u_2|_{\mathcal{H}^{-\beta}}^2)|u_2|_{\mathcal{H}^{1}}^2+L\lambda_{N}^{\beta}+\frac{L^{1+\beta}}{1+\beta}\left(\frac{m}{4}\right)^{-\beta}\right)|p|_{\mathcal{H}^{-\beta}}^2ds.
\end{align}
Indeed, we deduce that
\begin{align*}
|\bar{a}(w)(Au_2,q)|&\leq\ell_a(|u_1|_{\mathcal{H}^{-\beta}}+|u_2|_{\mathcal{H}^{-\beta}})|w|_{\mathcal{H}^{-\beta}}|u_2|_{\mathcal{H}^{1}}|q|_{\mathcal{H}^{1}}\\
&\leq\frac{m}{4}|q|_{\mathcal{H}^{1}}^2+\frac{\ell_a^2}{m}(|u_1|_{\mathcal{H}^{-\beta}}+|u_2|_{\mathcal{H}^{-\beta}})^2|u_2|_{\mathcal{H}^{1}}^2|w|_{\mathcal{H}^{-\beta}}^2
\\
&\leq\frac{m}{4}|q|_{\mathcal{H}^{1}}^2+\frac{2\ell_a^2}{m}(|u_1|_{\mathcal{H}^{-\beta}}^2+|u_2|_{\mathcal{H}^{-\beta}}^2)|u_2|_{\mathcal{H}^{1}}^2|p|_{\mathcal{H}^{-\beta}}^2.
\end{align*}
and 
\begin{align*}
|(F(u_1)-F(u_2),q)|&\leq L|w|^2_{\mathcal{H}}\\
&\leq L(|p|_{\mathcal{H}^{-\beta}}|p|_{\mathcal{H}^{\beta}}+|q|_{\mathcal{H}^{-\beta}}^{\frac{2}{1+\beta}}|q|_{\mathcal{H}^{1}}^{\frac{2\beta}{1+\beta}})\\
&\leq L\lambda_{N}^{\beta}|p|_{\mathcal{H}^{-\beta}}^2+
\frac{m}{4}|q|_{\mathcal{H}^{1}}^2+\frac{L^{1+\beta}}{1+\beta}\left(\frac{m}{4}\right)^{-\beta}|p|_{\mathcal{H}^{-\beta}}^2.
\end{align*}
Therefore, multiplying \eqref{q} by $A^{-\beta}q$, we have
\begin{align*}
 \frac12\frac{\mathrm{d}}{\mathrm{d}t}|q|_{\mathcal{H}^{-\beta}}^2+\frac{m}{2}|q|_{\mathcal{H}^{1}}^2
&\leq\left(\frac{2\ell_a^2}{m}(|u_1|_{\mathcal{H}^{-\beta}}^2+|u_2|_{\mathcal{H}^{-\beta}}^2)|u_2|_{\mathcal{H}^{1}}^2+L\lambda_{N}^{\beta}+
\frac{L^{1+\beta}}{1+\beta}\left(\frac{m}{4}\right)^{-\beta}\right)|p|_{\mathcal{H}^{-\beta}}^2.
\end{align*}
Integrating the above inequality from $t$ to $0$, since $w(t)\in K^{+}$ for $t\in[-T,0]$, we infer that \eqref{2.29} is true.
Combining \eqref{2.28} and \eqref{2.29}, we deduce that
\begin{align}\label{2.446}
&\frac{1}{2}|p(t)|_{\mathcal{H}^{-\beta}}^2\\
&\leq|p(0)|_{\mathcal{H}^{-\beta}}^2+\int_{t}^{0}\left(\frac{2\ell_a^2}{m}(|u_1|_{\mathcal{H}^{-\beta}}^2+|u_2|_{\mathcal{H}^{-\beta}}^2)|u_2|_{\mathcal{H}^{1}}^2+L\lambda_{N}^{\beta}+
\frac{L^{1+\beta}}{1+\beta}\left(\frac{m}{4}\right)^{-\beta}\right)|p|_{\mathcal{H}^{-\beta}}^2ds \nonumber\\
&\quad+\int_{t}^{0}\lambda_{N}^{1+\beta}\left(2M+\sqrt{2}\ell_a(|u_1|_{\mathcal{H}^{-\beta}}^2+3|u_2|_{\mathcal{H}^{-\beta}}^2)+2L\lambda_{N}^{-1}(1+\frac{1}{m}L\lambda_{N+1}^{-1})\right)|p|_{\mathcal{H}^{-\beta}}^2ds.\nonumber
\end{align}
Moreover, taking the inner product of \eqref{backward solution} with $A^{-\beta}P_Nu$, we obtain
\begin{align*}
\frac{1}{2}\frac{\mathrm d}{\mathrm dt}
|P_Nu|_{\mathcal H^{-\beta}}^2
\geq-(M\lambda_N^{1+\beta}+\lambda_N^\beta)
|P_Nu|_{\mathcal H^{-\beta}}^2
-\frac{K^2}{4}.
\end{align*}
Set $c_N:=M\lambda_N^{1+\beta}+\lambda_N^\beta.$
Multiplying the above inequality by $2e^{2c_Nt}$ and integrating over $[t,0]$, where $t\in[-T,0]$, yields
\begin{align}\label{2.334}
|P_Nu(t)|_{\mathcal H^{-\beta}}^2
&\leq e^{2c_NT}(|P_Nu(0)|_{\mathcal H^{-\beta}}^2+\frac{K^2}{4c_N}).
\end{align}
On the other hand, by Proposition \ref{well poseness}, for every $T>0$ and $i=1,2$, we have
\begin{align*}
&|u_i|_{L^\infty((-T,0);\mathcal H^{-\beta})}^2
+|u_i|_{L^2((-T,0);\mathcal H^1)}^2\leq
C(m,T,K)+|P_Nu_i(-T)|_{\mathcal H^{-\beta}}^2.
\end{align*}
Using \eqref{2.334} with $t=-T$, we therefore obtain
\begin{align*}
&|u_i|_{L^\infty((-T,0);\mathcal H^{-\beta})}^2
+|u_i|_{L^2((-T,0);\mathcal H^1)}^2\leq C(T,m,M,\beta,\lambda_N,K,
|P_Nu_i(0)|_{\mathcal H^{-\beta}})<\infty.
\end{align*}
Thus, by applying Gronwall's inequality to \eqref{2.446}, we conclude that
\begin{align*}
|p(t)|_{\mathcal H^{-\beta}}^2
&\leq
C(T,\ell_a,L,m,M,\beta,\lambda_N,\lambda_{N+1},K,
|P_Nu_i(0)|_{\mathcal H^{-\beta}})
|p(0)|_{\mathcal H^{-\beta}}^2,
\quad t\in[-T,0].
\end{align*}
Since $p(0)=0$, it follows that $p(-T)=0.$
On the other hand, $q(-T)=0$. Hence
\[
v_1^+-v_2^+=w(-T)=p(-T)+q(-T)=0,
\]
and therefore $v_1^+=v_2^+.$ Thus, $G_T$ is injective. By the theorem of invariance of domain \cite{Brouwer,James2000}, $G_T(\mathcal{H}_{+}^{-\beta})$ is an open set and $G_T:\mathcal{H}_{+}^{-\beta}\to G_T(\mathcal{H}_{+}^{-\beta})$ is a homeomorphism.

We now show that $G_T(\mathcal{H}_{+}^{-\beta})$ is closed. Let $\{v_n^+\}_{n\in\mathbb{N}}\subset\mathcal{H}_{+}^{-\beta}$ be a sequence such that $G_{T}(v_n^+)$ converges to $v^+\in\mathcal{H}_{+}^{-\beta}$. Thus, the sequence $\{G_{T}(v_n^+)\}_{n\in\mathbb{N}}$ is bounded. It follows from \eqref{2.334} that $\{v_n^+\}_{n\in\mathbb{N}}$ is bounded. Since ${\mathcal{H}_{+}^{-\beta}}$ is finite-dimensional, $\{v_n^+\}_{n\in\mathbb{N}}$ has a convergent subsequence (still denoted by $\{v_n^+\}_{n\in\mathbb{N}}$) that converges to some $z^+\in\mathcal{H}_{+}^{-\beta}$. From the continuity of $G_T$, we see that $G_T(v_n^+)\to G_T(z^+)$ as $n\to \infty,$ which gives $G_T(z^+)=v^+$. Hence, $G_T(\mathcal{H}_{+}^{-\beta})$ is closed and $G_T(\mathcal{H}_{+}^{-\beta})=\mathcal{H}_{+}^{-\beta}$. Consequently, $G_T$ is a homeomorphism from $\mathcal{H}_{+}^{-\beta}$ to itself. As a result, $G_T(v^+)=u_0^+$ is uniquely solvable for all $v^+\in \mathcal{H}_{+}^{-\beta}$. It follows that $u_{T,u_0^+}(t):=S_a(t+T) G_T^{-1}\left(u_0^+\right)$ is the unique solution of \eqref{backward solution}.\\
{\textbf{Step 2.}} 
We intend to prove 
\begin{align}\label{limit}
{u}_{\infty,u_0^+}(t):=\lim _{T \rightarrow \infty} u_{T,u_0^+}(t)
\end{align}
exists for all $t \leq 0$ and it is a backward solution of the problem \eqref{backward solution} with $T=\infty$ and $P_N u(0)=u_0^+$.

Indeed, suppose $T_2>T_1>0$, and set $w(t)=u_{T_1,u_0^+}(t)-u_{T_2,u_0^+}(t)$, $p(t)=P_N w(t)$ and $q(t)=Q_N w(t)$,
for $t \in\left[-T_1, 0\right]$. We note that
$$
p(0)=P_N u_{T_1}(0)-P_N u_{T_2}(0)=0,
$$
then $w(0) \notin K^{+}$. By the invariance of $K^{+}$, we find that
 $$
 |q(t)|_{\mathcal{H^{-\beta}}}\geq|p(t)|_{\mathcal{H^{-\beta}}},\quad \text{ for }t \in\left[-T_1, 0\right], 
 $$ and thus $$
 |w(t)|_{\mathcal{H^{-\beta}}}\leq\sqrt{2}|q(t)|_{\mathcal{H^{-\beta}}},
 \quad \text{ for }t \in\left[-T_1, 0\right].
 $$ 
 Since $Q_N u_{T_1,u_0^+}\left(-T_1\right)=0$ and the squeezing property, we obtain
\begin{align}\label{2.31}
|w(t)|_{\mathcal{H^{-\beta}}}&\leq Ce^{-\gamma\left(t+T_1\right)}|w\left(-T_1\right)|_{\mathcal{H^{-\beta}}} \nonumber\\
&\leq\sqrt{2}C e^{-\gamma\left(t+T_1\right)}|q\left(-T_1\right)|_{\mathcal{H^{-\beta}}}\\
&\leq\sqrt{2}C e^{-\gamma\left(t+T_1\right)}|Q_N u_{T_2,u_0^+}\left(-T_1\right)|_{\mathcal{H^{-\beta}}} \nonumber\\
&\leq \sqrt{2}C\lambda_{N+1}^{-\frac{\beta}{2}}e^{-\gamma\left(t+T_1\right)}|Q_N u_{T_2,u_0^+}\left(-T_1\right)|_{\mathcal{H}}, \quad \text{ for }t \in\left[-T_1, 0\right].\nonumber
\end{align}
Moreover, by \eqref{dissipative in H}, we deduce that
\begin{align*}
|u_{T_1,u_0^+}(t)-u_{T_2,u_0^+}(t)|_{\mathcal{H^{-\beta}}}\leq
\sqrt{2}CR_1\lambda_{N+1}^{-\frac{\beta}{2}}e^{-\gamma\left(t+T_1\right)}, \quad\text{for}~~t \in[-T_1, 0].
\end{align*}
In addition, for any $t^*>0$, as $T_1\rightarrow \infty$, we see that
$$
|u_{T_1,u_0^+}(t)-u_{T_2,u_0^+}(t)|_{\mathcal{H^{-\beta}}}\rightarrow 0, \quad \text {uniformly on }\left[-t^*, 0\right].
$$
This implies that $\{u_{T,u_0^+}:T>0\}$ is a Cauchy sequence in $C\left(\left[-t^*, 0\right] ; \mathcal{H}^{-\beta}\right)$ for any $t^*>0$. Hence, the limit \eqref{limit} exists and ${u}_{\infty,u_0^+}(t)$ is a backward solution of \eqref{backward solution}. Moreover, we observe that the high-mode component of $u_{\infty,u_0^+}(0)$ is uniformly bounded with respect to $u_0^+$. Indeed, for every $T>0$, by construction, $Q_Nu_{T,u_0^+}(-T)=0.$
Applying \eqref{dissipative in H} to $u_{T,u_0^+}(t-T)$, $t\in[0,T]$, we obtain
\[
|Q_Nu_{T,u_0^+}(0)|_{\mathcal H}\leq R_1.
\]
Consequently,
\[
|Q_Nu_{T,u_0^+}(0)|_{\mathcal H^{-\beta}}
\leq\lambda_{N+1}^{-\frac{\beta}{2}}|Q_Nu_{T,u_0^+}(0)|_{\mathcal H}
\leq
\lambda_{N+1}^{-\frac{\beta}{2}}R_1.
\]
Since
\[
u_{T,u_0^+}(0)\to u_{\infty,u_0^+}(0)
\quad\text{in }\mathcal H^{-\beta}\quad \text{as }T\to\infty,
\]
passing to the limit gives
\begin{equation}\label{uniform high mode}
|Q_Nu_{\infty,u_0^+}(0)|_{\mathcal H^{-\beta}}
\leq\lambda_{N+1}^{-\frac{\beta}{2}}R_1,
\quad
u_0^+\in\mathcal H_+^{-\beta}.
\end{equation}
{\textbf{Step 3.}} 
It follows from Proposition \ref{well poseness} that $u_{\infty,u_0^+}(t)$ has a unique extension, denoted by $u_{\infty}(t)$ that solves \eqref{backward solution} for all $t\in\mathbb{R}$. Let $\mathcal{S}$ denote the set of all such solutions.
By construction, $\mathcal{S}$ is invariant with respect to $t$, that is,
$$
T(h) \mathcal{S}=\mathcal{S}, \quad T(h)u_{\infty}(t):=u_{\infty}(t+h), \quad h \in \mathbb{R}.
$$
Define $\Phi: {\mathcal{H}_{+}^{-\beta}}\rightarrow \mathcal{H}^{-\beta}_{-}$ by
	$$
	\Phi\left(u_0^+\right):=Q_N u_{\infty}(0),
	$$
with $P_N u_{\infty}(0)=u_0^+$. It is obvious that for any ${u}_{\infty}^1$ and ${u}_{\infty}^2 \in \mathcal{S}$,
$u_{\infty}^1(t)-u_{\infty}^2(t) \in K^{+}$ for all $t \in \mathbb{R}$. As a direct consequence, $\Phi$ is Lipschitz continuous. Let us consider the set 
$$
\mathcal{M}:=\{u_0^++\Phi\left(u_0^+\right), u_0^+ \in {\mathcal{H}_{+}^{-\beta}}\},
$$ 
which is the desired inertial manifold. The invariance of $\mathcal{M}$ under the semigroup $S_a(t)$ follows immediately from the invariance of $\mathcal{S}$.

It remains to verify the exponential tracking property \eqref{tracking}. Indeed, let $u(t), t \geq0$, be a forward solution of \eqref{nonlocal problem}. We note that $P_Nu(t)\in\mathcal{H}_{+}^{-\beta}$, then, 
for any $T>0$, there exists $u_{\infty,T}(t)\in\mathcal{S}$ such that
	\begin{align*}
	P_N u(T)=P_N u_{\infty,T}(T),
	\end{align*}
	which implies that 
	$$
	|Q_Nu(T)-Q_Nu_{\infty,T}(T)|_{\mathcal{H^{-\beta}}}>0=|P_N u(T)-P_N u_{\infty,T}(T)|_{\mathcal{H^{-\beta}}}.
	$$
	By the squeezing property, we have
	\begin{align}\label{2.8}
	|P_N\left(u(t)-u_{\infty,T}(t)\right)|_{\mathcal{H^{-\beta}}}\leq|Q_N\left(u(t)-u_{\infty,T}(t)\right)|_{\mathcal{H^{-\beta}}},\quad \text{for}\ t\in[0, T],
	\end{align}
	and 
	\begin{align}\label{2.339}
	|u(t)-u_{\infty,T}(t)|_{\mathcal{H^{-\beta}}}&\leq C\mathrm{e}^{-\gamma t}|u(0)-u_{\infty,T}(0)|_{\mathcal{H^{-\beta}}}\nonumber\\
  &\leq C\mathrm{e}^{-\gamma t}|Q_N(u(0)-u_{\infty,T}(0))|_{\mathcal{H^{-\beta}}}.
	\end{align}
Since $u_{\infty,T}(0)\in\mathcal M$, we have
$Q_Nu_{\infty,T}(0)=\Phi(P_Nu_{\infty,T}(0)).$
Therefore, by \eqref{uniform high mode},
\[
|Q_Nu_{\infty,T}(0)|_{\mathcal H^{-\beta}}
\leq\lambda_{N+1}^{-\frac{\beta}{2}}R_1,
\qquad T>0.
\]
Thus, $\{Q_Nu_{\infty,T}(0)\}$ is uniformly bounded in $\mathcal H^{-\beta}$ with respect to $T$. Moreover, in view of \eqref{2.8}, we know that
 \begin{align*}
 |P_Nu_{\infty,T}(0)|_{\mathcal{H^{-\beta}}}\leq2|u(0)|_{\mathcal{H^{-\beta}}}+\lambda_{N+1}^{-\frac{\beta}{2}}R_1,
 \end{align*}
 that is, $\{P_Nu_{\infty,T}(0)\}$ is also uniformly bounded in $\mathcal{H}_{+}^{-\beta}$ with respect to $T$. 
 Since $\mathcal{H}_{+}^{-\beta}$ is finite-dimensional, $\{P_Nu_{\infty,T}(0)\}$ has a convergent subsequence (still denoted by $\{P_Nu_{\infty,T}(0)\}$) that converges to $u_{\infty}^+\in\mathcal{H}_{+}^{-\beta}$. 
In addition, thanks to $u_{\infty,T}(0)\in\mathcal{M}$ and the Lipschitz continuity of $\Phi$, we deduce that
\begin{align}\label{2.440}
{u}_{\infty,T}(0) \rightarrow {u}_{\infty}^++\Phi({u}_{\infty}^+):={u}_{\infty}(0) \in \mathcal{M}\quad\ \text{as}\ T\rightarrow \infty,
\end{align}
and the corresponding trajectory ${u}_{\infty}(t) \in \mathcal{S}$. Since ${u}_{\infty,T}$ and ${u}_{\infty}$ are both solutions of \eqref{nonlocal problem}, due to \eqref{continuous}, for any $t\geq0,$
we obtain that
	$$
	|{u}_{\infty,T}(t)-{u}_{\infty}(t)|_{\mathcal{H^{-\beta}}}^2 \leq e^{C(m,\ell_a,\lambda_1,L,K,|u_{\infty}(0)|_{{\mathcal{H}}^{-\beta}},|u_{\infty,T}(0)|_{{\mathcal{H}}^{-\beta}},t)t} |{u}_{\infty,T}(0)-{u}_{\infty}(0)|_{\mathcal{H^{-\beta}}}^2,
	$$
 which tends to 0 as $T\to\infty.$
	It follows from \eqref{2.339}, \eqref{2.440} and the above inequality that
	$$
	|u(t)-{u}_{\infty}(t)|_{\mathcal{H^{-\beta}}}\leq Ce^{-\gamma t}|u(0)-{u}_{\infty}(0)|_{\mathcal{H^{-\beta}}}, \quad\text { for } t \geq 0.
	$$
	This completes the proof.
\end{proof}	
\subsection{A Spectral Gap Condition for the Modified Problem}

Theorem \ref{abstract theorem} provides an abstract criterion for the existence of inertial manifolds in terms of the strong squeezing property. However, for the nonlocal problem \eqref{nonlocal problem}, it is not straightforward to verify this property directly through the strong cone condition as in the semilinear case, see \cite{li2024,Zelik} for more details. Indeed, if $u_1$ and $u_2$ are two solutions of \eqref{nonlocal problem}, then, in the differential inequality for the cone functional associated with $w=u_1-u_2$, an additional term of the form
\[
|(a(|u_1|_{H^{-\beta}}^2)-a(|u_2|_{H^{-\beta}}^2))
(Au_2,q-p)|
\]
appears. This term is caused by the dependence of the diffusion coefficient on the solution itself and cannot, in general, be controlled only by the spectral gap of $A$ and the Lipschitz constant of $F$. Therefore, a direct verification of the strong cone condition for \eqref{nonlocal problem} seems difficult.

To overcome this difficulty, we use the fact that the nonlocal coefficient is a positive scalar function and introduce a suitable change of the time scale. Since we are mainly interested in the long-time behavior of the solutions, we first suitably modify the nonlocal term outside an absorbing ball. After changing the scale of time, the modified nonlocal problem is transformed into a semilinear parabolic equation with the fixed principal part $A^{1+\beta}$. This enables us to apply the classical theory of inertial manifolds for semilinear parabolic equations. In particular, by estimating the Lipschitz constant of the transformed nonlinearity, we derive an explicit spectral gap condition which guarantees the existence of an inertial manifold for the modified problem.
To do this, we choose the following $\mathscr{C}^{\infty}$ cut-off function $\theta(s):\mathbb{R}^{+}\to[0,1]$ such that
\begin{equation}\label{2.551}
	\left\{\begin{aligned}
		&\theta(s)=1 \quad \text {for } 0 \leq s\leq1,\\
		&\theta(s)=0 \quad \text {for } s \geq 2,\\
		&\displaystyle\sup_{s \geq 0} |\theta^{\prime}(s)| \leq 2.
	\end{aligned}\right.
\end{equation} 
Set $\theta_{R}(s)=\theta(s/R)$. 
It follows from Proposition \ref{global attractor} that there exists a constant $\rho>0$ such that
$$
\mathcal{B_{\rho}}:=\{u\in\mathcal{H}^{-\beta}; |u|_{\mathcal{H}^{-\beta}}\leq
\rho\}
$$
is an absorbing ball for the semigroup $S_a(t)$ associated with problem \eqref{nonlocal problem}.
Then, we define a function $\tilde{a}(u):\mathcal{H}^{-\beta}\to\mathbb{R}^+$ by 
$$
\tilde{a}(u)=a(\theta_{\rho}(|u|_{\mathcal{H}^{-\beta}})|u|_{\mathcal{H}^{-\beta}}^2).
$$
Hence, $\tilde a$ is globally Lipschitz continuous on
$\mathcal H^{-\beta}$, and we may take $\ell_{\tilde a}:=12\ell_a\rho$ as its Lipschitz constant. In addition, in the absorbing ball $\mathcal{B}_{\rho}$,
problem \eqref{nonlocal problem} coincides with the following modified problem:
\begin{equation}\label{cut-off problem}
	\left\{\begin{aligned}
		&\partial_t u+\tilde a(u)A^{1+\beta}u=A^{\beta}F(u),\ t>0,\\
		&u(0)=u_0,
	\end{aligned}\right.
\end{equation} 
that is, problem \eqref{cut-off problem} contains all asymptotic behavior of \eqref{nonlocal problem}. Consequently, it is sufficient to investigate the existence of inertial manifolds for \eqref{cut-off problem}. 

As shown in \cite{chipot2015,chipot2003}, by changing the scale of time to
\begin{align*}
\tau=\mathcal{T}_u(t):=\int_0^t \tilde a(u(r, u_0))dr,
\end{align*}
it follows that $z(\tau,u_0):=u(t,u_0)$ is the solution of the following auxiliary semilinear parabolic problem:
\begin{equation}\label{changed cut-off problem}
	\left\{\begin{aligned}
		&\partial_{\tau}z+A^{1+\beta}z=A^{\beta}\tilde F(z),\ \tau>0,\\
		&z(0)=u_0,
	\end{aligned}\right.
\end{equation} 
where $\widetilde F(z):=\frac{F(z)}{\tilde a(z)}.$
The mapping $\widetilde F$ is globally bounded and globally Lipschitz on $\mathcal{H}$, as stated in the following lemma. Since the proof is standard, we omit it here and refer the reader to \cite{Robinson,Xiaoqing} for more details.

\begin{lemma}
Under the above assumptions, $\widetilde F:\mathcal H\to\mathcal H$ is globally bounded and globally Lipschitz continuous. More precisely,
\begin{equation}\label{tildeF}
	\begin{aligned}
	|\widetilde F(z)|_{\mathcal H}&\leq\frac{K}{m},\quad z\in\mathcal{H},\\
	|\widetilde F(z_1)-\widetilde F(z_2)|_{\mathcal H}
   &\leq(\frac{L}{m}+\frac{12K\ell_a\rho\lambda_1^{-\beta/2}}{m^2})|z_1-z_2|_{\mathcal H}, \quad z_1,z_2\in\mathcal{H}.
	\end{aligned}
\end{equation}
\end{lemma}

In view of the preceding analysis, we formulate the following spectral gap condition, which ensures the existence of an inertial manifold for \eqref{cut-off problem}.
\begin{proposition}	\label{main result}
Under the above assumptions, assume that there exists $N \in \mathbb{Z}^{+}$ such that the following spectral gap condition holds:
\begin{align}\label{spectral gap}
	\frac{\lambda_{N+1}^{1+\beta}-\lambda_{N}^{1+\beta}}{\lambda_{N+1}^{\beta}+\lambda_{N}^{\beta}}>
    \frac{L}{m}+\frac{12K\ell_a\rho\lambda_1^{-\beta/2}}{m^2},
\end{align}
 Then, problem \eqref{cut-off problem} possesses an $N$-dimensional inertial manifold.
\end{proposition}
\begin{proof}
It follows from Proposition 4.1 in \cite{li2024} that \eqref{changed cut-off problem} possesses an $N$-dimensional inertial manifold $\mathcal{M}$ associated with the semigroup $\{S(\tau):\tau\geq 0\}$. We next show that $\mathcal{M}$ is also an $N$-dimensional inertial manifold for the semigroup $\{S_a(t):t\geq 0\}$ associated with \eqref{cut-off problem}.
It is clear that $\mathcal{M}$ is a finite-dimensional Lipschitz manifold.
For any $u_0\in\mathcal{H}^{-\beta}$, let $u(\cdot,u_0)$ and $z(\cdot,u_0)$ denote the solutions to \eqref{cut-off problem} and \eqref{changed cut-off problem}, respectively. Define
\begin{align}
\tau_{u_0}(t):=\int_0^t \tilde a(u(r,u_0))dr, \quad t \geq 0,
\end{align}
\begin{align}
t_{u_0}(\tau):=\int_0^\tau \frac{1}{\tilde a(z(s,u_0))} ds,\quad \tau\geq0.
\end{align}
Thus, 
\begin{equation}\label{orbit relation}
S_a(t)u_0=S(\tau_{u_0}(t))u_0,
\quad
S(\tau)u_0=S_a(t_{u_0}(\tau))u_0.
\end{equation}
Since for any $s\geq0$, $\tilde a(s)\in [m,M]\subset(0,+\infty)$, we have
$$
m t \leq \tau_{u_0}(t) \leq M t, \quad \frac{\tau}{M} \leq t_{u_0}(\tau) \leq \frac{\tau}{m}.
$$
\medskip
\noindent\textbf{Step 1. Invariance of $\mathcal M$ under $S_a(t)$.}
It follows from \eqref{orbit relation} and 
$S(\tau)\mathcal M=\mathcal M$ for every $\tau\geq0$ that $S_a(t)\mathcal{M}\subset \mathcal{M}$ for all $t\geq0.$ 
We next prove that $\mathcal{M}\subset S_a(t)\mathcal{M}$ for all $t\geq0.$ 
Fix $t>0$ and $u_0\in\mathcal M$, and set $\tau_*:=Mt.$
Since $S(\tau_*)\mathcal M=\mathcal M$, there exists $z_*\in\mathcal M$ such that $S(\tau_*)z_*=u_0.$
For $0\leq r\leq\tau_*$, define
\[
\Theta(r):=
\int_{\tau_*-r}^{\tau_*}
\frac{1}{\tilde a(z(s,z_*))}\,ds.
\]
The function $\Theta$ is continuous and strictly increasing, $\Theta(0)=0$, and
\[
\Theta(\tau_*)
=\int_0^{\tau_*}\frac{1}{\tilde a(z(s,z_*))}\,ds
\geq\frac{\tau_*}{M}=t.
\]
Hence, there exists $r_*=r_*(t)\in[0,\tau_*]$ such that
$\Theta(r_*)=t.$
Set
\[
z_0:=S(\tau_*-r_*)z_*\in\mathcal M.
\]
Then,
$$
S(r_*)z_0=S(r_*)S(\tau_*-r_*)z_*=S(\tau_*)z_*=u_0,
$$
and 
\[
z(s,z_0)=z(s+\tau_*-r_*,z_*),\quad 0\leq s\leq r_*,
\]
Therefore,
\begin{align*}
t_{z_0}(r_*)
&=\int_0^{r_*}\frac{1}{\tilde a(z(s,z_0))}\,ds\\
&=\int_{0}^{r_*}\frac{1}{\tilde a(z(s+\tau_*-r_*,z_*))}\,ds\\
&=\int_{\tau_*-r_*}^{\tau_*}\frac{1}{\tilde a(z(s,z_*))}\,ds=t,
\end{align*}
which implies that $
S_a(t)z_0=S(r_*)z_0=u_0.$ Hence, the desired result follows.

\medskip
\noindent\textbf{Step 2. Exponential tracking under $S_a(t)$.}
Fix $u_0\in\mathcal H^{-\beta}$. Since $\mathcal M$ is an inertial manifold for $S(\tau)$, there exist $\bar v_0\in\mathcal M$, constants $C_0>0$ and $\delta>0$, such that
\begin{equation}\label{semi tracking}
|S(\tau)u_0-S(\tau)\bar v_0)|_{\mathcal H^{-\beta}}
\leq C_0e^{-\delta\tau},
\qquad \tau\geq0.
\end{equation}
Let $D(\tau):=t_{\bar v_0}(\tau)-t_{u_0}(\tau)$, it follows from \eqref{solution estimate} and \eqref{semi tracking} that
\begin{align*}
|D^{\prime}(\tau)|=|\frac{1}{\tilde a(z(\tau,\bar v_0))}-\frac{1}{\tilde a(z(\tau,u_0))}|\leq\frac{\ell_{\tilde a}}{m^2}C(u_0,\bar v_0)e^{-\delta\tau},\quad \tau\geq0,
\end{align*}
where $D^{\prime}(\tau)$ denotes the derivative of $D(\tau)$ with respect to $\tau.$ Therefore, the following limit exists,
\begin{equation}\label{D infinity}
D_\infty:=\lim_{\tau\to+\infty}D(\tau),
\end{equation}
and
\begin{align}\label{clock tail}
|D_\infty-D(\tau)|&\leq\int_\tau^\infty|\frac{1}{\tilde a(z(s,\bar v_0))}-\frac{1}{\tilde a(z(s,u_0))}|ds\nonumber\\ 
&\leq\frac{\ell_{\tilde a}}{\delta m^2}C(u_0,\bar v_0)e^{-\delta\tau},\quad \tau\geq0.
\end{align}
If $D_\infty\geq0$, the function $t_{\bar v_0}(\tau)$ is continuous, strictly increasing and maps $[0,\infty)$ onto $[0,\infty)$. Hence there exists a unique $\theta\geq0$ such that $t_{\bar v_0}(\theta)=D_\infty.$
Set
\[
v_0:=z(\theta,\bar v_0)\in\mathcal M.
\]
For $t\geq0$, let $\sigma(t)\geq\theta$ be determined by
$t_{\bar v_0}(\sigma(t))=t+D_\infty.$ Then, 
the inverse time transformation and the semigroup property give
\begin{align*}
S(\sigma(t))\bar v_0&=S_a(t+D_{\infty})\bar v_0\\
&=S_a(t)S_a(D_{\infty})\bar v_0\\
&=S_a(t)S(\theta)\bar v_0=S_a(t)v_0,\qquad t\geq0.
\end{align*}
If $D_\infty<0$, then, since $\mathcal{M}$ is invariant under $S_a(t)$, there exists $v_0\in\mathcal{M}$ such that
$S_a(-D_\infty)v_0=\bar v_0$. Thus, for every $t\geq-D_\infty$, the semigroup property gives
\[
S_a(t)v_0=S_a(t+D_{\infty})\bar v_0.
\]
Let $\sigma(t)\geq0$ be defined by $t_{\bar v_0}(\sigma(t))=t+D_\infty,$
then
\begin{equation}\label{physical tracking orbit negative}
S(\sigma(t))\bar v_0=S_a(t+D_{\infty})\bar v_0=S_a(t) v_0,\qquad t\geq -D_{\infty}.
\end{equation}
Hence, there exists $v_0\in\mathcal M$ such that, 
\begin{equation}\label{unified phase}
S_a(t)v_0=S(\sigma(t))\bar v_0, \quad t\geq |D_\infty|,
\end{equation}
where $t_{\bar v_0}(\sigma(t))=t+D_\infty.$ Moreover, set $\tau(t):=\tau_{u_0}(t),$ we obtain
\begin{align*}
t_{\bar v_0}(\sigma(t))-t_{\bar v_0}(\tau(t))
=t+D_\infty-t_{\bar v_0}(\tau(t))=D_\infty-D(\tau(t)).
\end{align*}
Since
\[
\frac{d}{d\tau}t_{\bar v_0}(\tau)
=\frac1{\tilde a(z(\tau,\bar v_0))}\geq\frac1M,
\]
thus, it follows from \eqref{clock tail} that
\begin{align}\label{time synchronization}
|\sigma(t)-\tau(t)|&\leq M|t_{\bar v_0}(\sigma(t))-t_{\bar v_0}(\tau(t))|\leq\frac{\ell_{\tilde a}M}{\delta m^2}C(u_0,\bar v_0)e^{-\delta\tau(t)},\quad t\geq |D_{\infty}|.
\end{align}
In addition, it is clear that there exists $\Phi:\mathcal{H}^{-\beta}_+\to\mathcal{H}^{-\beta}_-$ which is Lipschitz continuous with Lipschitz constant $\operatorname{Lip}\Phi$,
such that
$$
\mathcal{M}=\{p+\Phi(p):p\in\mathcal{H}^{-\beta}_+\}.
$$
Hence, on this manifold $\mathcal{M}$, the low-mode component $p(\tau):=P_NS(\tau)\bar v_0$ satisfies the finite-dimensional system
$$
\frac{dp}{d\tau}=-A^{1+\beta}p+P_NA^{\beta}\widetilde F(p+\Phi(p)).
$$
Therefore, it follows from \eqref{solution estimate} that 
$$
|\frac{dp}{d\tau}(\tau)|_{\mathcal{H}^{-\beta}}
\leq\lambda_N^{1+\beta}|p(\tau)|_{\mathcal{H}^{-\beta}}+
\lambda_N^{\frac{\beta}{2}}\frac{K}{m}\leq C(\bar v_0,m,\lambda_{N},K):=C_p,\quad \tau\geq0.
$$
It follows that, for any $\tau_1,\tau_2\geq0$,
$$
|p(\tau_1)-p(\tau_2)|_{\mathcal{H}^{-\beta}}\leq C_p|\tau_1-\tau_2|.
$$
By using the Lipschitz continuity of $\Phi$, we have
\begin{align}\label{continuous t}
|S(\tau_1)\bar v_0-S(\tau_2)\bar v_0|_{\mathcal{H}^{-\beta}}
&\leq
|p(\tau_1)-p(\tau_2)|_{\mathcal{H}^{-\beta}}
+|\Phi(p(\tau_1))-\Phi(p(\tau_2))|_{\mathcal{H}^{-\beta}}
\nonumber\\
&\leq(1+\operatorname{Lip}\Phi)
|p(\tau_1)-p(\tau_2)|_{\mathcal{H}^{-\beta}}\nonumber\\
&\leq(1+\operatorname{Lip}\Phi)
C_p|\tau_1-\tau_2|.
\end{align}
Combining \eqref{semi tracking}, \eqref{unified phase}, \eqref{time synchronization} and \eqref{continuous t},
we obtain, 
\begin{align*}
|S_a(t)u_0-S_a(t)v_0|_{\mathcal H^{-\beta}}
&=|S(\tau(t))u_0-S(\sigma(t))\bar v_0|_{\mathcal H^{-\beta}}\\
&\leq |S(\tau(t))u_0-S(\tau(t))\bar v_0|_{\mathcal H^{-\beta}}+|S(\tau(t))\bar v_0-S(\sigma(t))\bar v_0|_{\mathcal H^{-\beta}}\\
&\leq C_0e^{-\delta\tau(t)}+(1+\operatorname{Lip}\Phi)C_p\frac{\ell_{\tilde a}M}{\delta m^2}C(u_0,\bar v_0)e^{-\delta\tau(t)}\\
&\leq C_1e^{-\delta mt},\quad t\geq|D_{\infty}|.
\end{align*}
Moreover, it follows from \eqref{solution estimate} that
\begin{align*}
|S_a(t)u_0-S_a(t)v_0|_{\mathcal H^{-\beta}}
&\leq C_2\leq C_2e^{\delta m|D_\infty|}e^{-\delta mt},\quad 0\leq t\leq|D_{\infty}|.
\end{align*}
Hence, the desired result follows.
\end{proof}
\section{Applications to the 2D Nonlocal Parabolic Equations}

This section is devoted to applying the theory developed for the abstract model
\eqref{nonlocal problem} to two-dimensional nonlocal parabolic equations.

In what follows, we set
$\Omega=(0,\pi)^2$
and take $\mathcal H:=X:=L^2(\Omega).$
Let $A:D(A)\subset L^2(\Omega)\to L^2(\Omega)$
be the operator defined by $Aw=-\Delta w,$
with homogeneous Dirichlet boundary conditions. It is well known that $A$ is a
positive self-adjoint operator with compact inverse and $-A$ generates an analytic semigroup on $L^2(\Omega)$. For $k=(k_1,k_2)\in(\mathbb Z^+,\mathbb Z^+),$ set
$e_k(x_1,x_2):=\frac{2}{\pi}\sin(k_1x_1)\sin(k_2x_2),$
and $\lambda_k:=|k|^2=k_1^2+k_2^2.$
Then $Ae_k=\lambda_ke_k,$
and the family $\{\frac{2}{\pi}\sin(k_1x_1)\sin(k_2x_2):
k_1,k_2\in\mathbb Z^+\}$ forms a complete orthonormal basis of $L^2(\Omega)$. For convenience, we enumerate the eigenvalues of $A$, counted with multiplicity,
in nondecreasing order as
$$
0<\lambda_1\leq\lambda_2\leq\cdots,
\quad
\lambda_j\to+\infty
\quad\text{as }j\to\infty,
$$
and denote by $\{e_j\}_{j=1}^\infty$ the corresponding
$L^2(\Omega)$ orthonormal eigenfunctions. Thus,
$$
Ae_j=\lambda_je_j,
\quad
j\in\mathbb Z^{+}.
$$
Accordingly, every $u\in L^2(\Omega)$ admits the expansion
$$
u=\sum_{j=1}^\infty u_je_j,
\qquad
u_j=(u,e_j).
$$
where $\|\cdot\|$ denotes the usual norm in $L^{2}(\Omega)$. Let $X^s$, $s\geq0$, denote the fractional power spaces associated with $A$,
endowed with the inner products and norms
$$
(u,v)_s=(A^su,A^sv),
\qquad
\|u\|_s=\|A^su\|,
$$
respectively. Equivalently,
$$
\|u\|_s^2=\sum_{j=1}^\infty
\lambda_j^{2s}|u_j|^2.
$$
In particular, $X^0=L^2(\Omega),
X^{\frac12}=H_0^1(\Omega), X^1=D(A)=H^2(\Omega)\cap H_0^1(\Omega).$
Moreover, since $\Omega\subset\mathbb R^2$, Weyl's asymptotic formula gives $\lambda_j\sim c_\Omega j$ as $j\to\infty,$ where the eigenvalues are counted with multiplicity.

\subsection{Diffusion Coefficient Depending on the $L^2$-Norm of the Solution}
Let us consider the following 2D nonlocal equation:
\begin{equation}\label{2Dnonlocal reaction}
	\begin{cases}
		u_t=a(\Vert u\Vert^2)\Delta u+\lambda f(u) &\text{
			in} ~ \Omega\times\mathbb{R}^+,\\
		u(x,t)=0 & \text{
			on} ~ \partial\Omega\times\mathbb{R}^+, \\
		u(x,0)=u_0& \text{
			in} ~ \Omega, \\
	\end{cases}
\end{equation}
where $\|\cdot\|:=\|\cdot\|_0$ denotes the usual norm in $L^{2}(\Omega)$, $\lambda>0$ is a parameter, $a:\mathbb{R^+}\to [m,M]\subset (0,+\infty)$ is a continuously differentiable function and is globally Lipschitz with Lipschitz constant $\ell_a$. 
The function $f$ satisfies $f(0)=0,$ and there exist positive constants $\eta, k, \alpha_1,\alpha_2>0$ such that
\begin{equation}
	\left\{\begin{aligned}
  & f(s)\in C^2(\mathbb{R}),\\		&f^{\prime}(s)\leq \eta,\quad \text{for}\ s\in{\mathbb{R}},\\
		&-k-\alpha_1|s|^p \leq f(s) s \leq k-\alpha_2|s|^p, \quad \text{for } p\geq2, s\in\mathbb{R}. 
	\end{aligned}\right.
\end{equation} 

In order to establish the existence of an inertial manifold for \eqref{2Dnonlocal reaction}, we need the following lemma. The proof is standard and thus omitted, see e.g., \cite{Carvalho-Langa-Robinson-13, Chueshov, Robinson,Temam}.
\begin{lemma}\label{lemma3.1}
	 Let the operator $A$, the functions $a(s)$ and $f(s)$ satisfy the above assumptions. Then for every $u_0 \in X$, problem \eqref{2Dnonlocal reaction} possesses a unique solution $u \in C([0, T] ; X) \cap L^2(0, T ; X^{\frac12})$ for all $T>0$. 
Moreover, the solution depends Lipschitz continuously on the initial data. In addition, problem \eqref{2Dnonlocal reaction} admits an absorbing ball $\mathcal{B}_{\rho_1}=:\{u\in {X}^{\frac12}; \|u\|_{\frac12}\leq\rho_1\}$, where $\rho_1>0$ is a constant. 
\end{lemma}

Let $S_a(t)$ denote the corresponding semigroup given by $u(t):=S_a(t)u_0$ for $t\geq0.$ From the Lemma \ref{lemma3.1}, the semigroup $S_a(t):X\to X$ has a global attractor $\mathcal{A}\subset {{X}^{\frac12}}$ in the phase space $X$.
Following the approach of Section 11.2.1 in
\cite{Robinson}, we conclude that the global attractor $\mathcal{A}$ is bounded in $L^{\infty}(\Omega)$. Moreover, we consider the following set:
$$
\mathcal{B}_0:=\left\{u \in X;\|u\| \leq \rho_0,\|u\|_{\infty} \leq \rho_{\infty}\right\},
$$
where $\rho_0$ and $\rho_{\infty}$ are bounds of $\mathcal{A}$ in $X$ and $L^{\infty}(\Omega)$, respectively. Proceeding as Section 15.4.1 in \cite{Robinson}, set 
$$
g(s):=\theta_{\rho_{\infty}}(|s|)\lambda f(s),\quad G(u)(x):=g(u(x)), \quad F(u):=\theta_{\rho_{0}}(\|u\|)G(u),
$$
where the cut-off function $\theta(s)$ is defined as in \eqref{2.551}.
It is clear that $F(\cdot):X\to X$ is now globally bounded and globally Lipschitz continuous, see Lemma 15.7 in \cite{Robinson} for more details.

In view of the above, we observe that the modified version of problem \eqref{2Dnonlocal reaction} fits the abstract model \eqref{nonlocal problem} with $\beta=0$, 
$\mathcal{H}=L^{2}(\Omega),$ and $A=-\Delta:
D(A)\subset L^2(\Omega) \rightarrow L^2(\Omega)$ with homogeneous Dirichlet boundary conditions. Thus, the desired model is constructed. As in the previous section, we now consider the following modified problem:
\begin{equation}\label{modified problem}
	\left\{\begin{aligned}
		&\partial_t u-\tilde a(u)\Delta u={F}(u),\ t>0,\\
		&u(0)=u_0,
	\end{aligned}\right.
\end{equation} 
where $\tilde{a}(u)=a(\theta_{\rho_0}(\|u\|)\|u\|^2).$
It remains to show that the spectral gap condition \eqref{spectral gap} holds, which is guaranteed by the following classical result from number theory.
The proof can be found in \cite{Richards}.

\begin{lemma}\label{two square}
	The sequence $\left\{s_n=k_1^2+k_2^2: k_1, k_2 \in \mathbb{Z}\right.$ and $\left.s_{n+1} \geq s_n\right\}$ satisfies
	$$
	\limsup _{n \rightarrow \infty} \frac{s_{n+1}-s_n}{\log s_n} \geq \delta,
	$$
	for some $\delta>0$.
\end{lemma}
Up to now, all conditions in Proposition \ref{main result} are fulfilled. As a consequence, we have the following result.
\begin{theorem} \label{main}
For some sufficiently large $N,$ the modified problem \eqref{modified problem} possesses an $N$-dimensional
inertial manifold with $\beta=0$, $\mathcal{H}=L^{2}(\Omega)$ and $A=-\Delta:D(A)\subset L^2(\Omega) \rightarrow L^2(\Omega)$ with homogeneous Dirichlet boundary conditions.   
\end{theorem}
\subsection{Diffusion Coefficient Depending on the $L 
^2$-Norm of the Gradient}

We next apply the abstract result to the following 2D nonlocal parabolic equation:
\begin{equation}\label{2DKir}
	\begin{cases}
		u_t=a(\Vert\nabla u\Vert^2)\Delta u+\lambda f(u) &\text{
			in} ~ \Omega\times\mathbb{R}^+,\\
		u(x,t)=0 & \text{
			on} ~ \partial\Omega\times\mathbb{R}^+, \\
		u(x,0)=u_0& \text{
			in} ~ \Omega, \\
	\end{cases}
\end{equation}
where $u_0 \in X^{\frac12}$, and the functions $a(s)$ and $f(s)$ satisfy the same assumptions as above. It is well known that problem \eqref{2DKir} possesses a global attractor $\mathcal{A}_1 \subset X^{\frac12}$, see \cite{Estefani,valeroJDDE} for more details. In addition, arguing as above, one can verify that $\mathcal{A}_1$ is bounded in $L^{\infty}(\Omega)$. However, the cut-off function introduced in Section~3.1 is no longer sufficient, since it does not ensure that the associated Nemytskii operator is globally Lipschitz from $X^{\frac12}$ into itself. This is because
$H_0^1(\Omega)\not\hookrightarrow L^\infty(\Omega)$
in two dimensions. To overcome this difficulty, we use the cut-off function introduced in \cite{ANA2018} to construct a suitable cut-off operator whose range is uniformly bounded in a function space with slightly higher regularity than $X^{\frac12}$. To do this, we need the following lemma.
\begin{lemma}\label{Kir-high-regularity}
Under the above assumptions, the global attractor $\mathcal A_1$ is bounded in $X^\alpha$, with
$\frac12<\alpha<\frac32.$
\end{lemma}

\begin{proof}
Let $u:\mathbb R\to\mathcal A_1$ be a complete trajectory. Since $\mathcal A_1$ is bounded in $X^{\frac12}=H_0^1(\Omega)$ and in $L^\infty(\Omega)$, there exists a constant $R>0$ such that
$$
\sup_{t\in\mathbb R}
(\|u(t)\|_{\frac12}+
\|u(t)\|_{L^\infty(\Omega)})
\leq R.
$$
Since $f(0)=0$ and $f\in C^2(\mathbb R)$, we have
$f(u(s))\in X^{\frac12}.$
Consequently, there exists a constant $C_1>0$, independent of $s\in\mathbb R$, such that
\begin{equation}\label{f-H12-bound}
\|f(u(s))\|_{\frac12}\leq C_1,
\qquad s\in\mathbb R.
\end{equation}
It follows from Remark \ref{remark2.2} that
\begin{align}\label{formula}
u(t)=e^{-A\int_{t-1}^t a(|\nabla u(r)|^2)dr}u(t-1)
+\lambda\int_{t-1}^te^{-A\int_s^t a(|\nabla u(r)|^2)dr}
f(u(s))ds,\quad t\in\mathbb R.
\end{align}
Using the standard smoothing estimate for the analytic semigroup generated by $-A$, for every
$\frac12\leq\alpha<\frac32$,
we obtain
$$
\|e^{-A\int_s^t a(\|\nabla u(r)\|^2)dr}v\|_\alpha
\leq C(m,\alpha) (t-s)^{-(\alpha-\frac12)}
\|v\|_{\frac12},\quad \text{for } v\in X^{\frac12}.
$$
Hence, by \eqref{formula},
\begin{align*}
\|u(t)\|_\alpha &\leq
C(m,\alpha)\|u(t-1)\|_{\frac12} +C(m,\alpha,\lambda)
\int_{t-1}^t(t-s)^{-(\alpha-\frac12)}
\|f(u(s))\|_{\frac12}ds\\
&\leq C(m,\alpha)R+C(m,\alpha,\lambda)C_1.
\end{align*}
Therefore, there exists a constant $\rho_{\alpha}>0$, independent of $t\in\mathbb R$ and the complete trajectory $u$, such that
\begin{align}\label{rhoalpaha}
\|u(t)\|_\alpha\leq \rho_{\alpha},
\quad t\in\mathbb R.
\end{align}
Hence, the desired result follows.
\end{proof}
Choose $\eta\in C^\infty(\mathbb R)$ such that
\[
 \eta(s)=s\quad\text{for }|s|\leq1,
 \qquad |\eta(s)|\leq2,
 \qquad \sup_{s\in\mathbb R}\bigl(|\eta'(s)|+|\eta''(s)|\bigr)<\eta_1.
\]
Fix $1<\alpha<\frac32,$ for any $u\in X^{\frac12}$,
$$
u=\sum_{j=1}^{\infty}u_je_j,
\qquad
u_j=(u,e_j),
$$
define the map $\mathcal W:X^{\frac12}\to X^{\frac12}$ by
\begin{equation}\label{preparation-map}
 \mathcal W(u):=\sum_{j=1}^{\infty}
 \rho_{\alpha}\lambda_j^{-\alpha}
 \eta(\frac{\lambda_j^{\alpha}u_j}{\rho_{\alpha}})e_j,
\end{equation}
where $\rho_\alpha$ is the same constant as in \eqref{rhoalpaha}.
\begin{lemma}\label{preparation-map-lemma}
The map $\mathcal W$ has the following properties:
\begin{enumerate}[(i)]
 \item $\mathcal W(u)=u$ for every $u\in\mathcal A_1$.
 \item $\mathcal W:X^{\frac12}\to X^{\frac12}$ is globally Lipschitz.
 \item There exist $s>\frac12$ and $p>2$ such that
 \begin{equation}\label{W-uniform-regularity}
 \sup_{u\in X^{\frac12}}
 \left(\|\mathcal W(u)\|_s
 +\|\mathcal W(u)\|_{L^\infty(\Omega)}
 +\|\nabla\mathcal W(u)\|_{L^p(\Omega)}\right)<\infty.
 \end{equation}
 \item $\mathcal W$ is Gateaux differentiable on $X^{\frac12}$ and
 \begin{equation}\label{W-derivative}
 \mathcal W'(u)h
 =\sum_{j=1}^{\infty}
 \eta'(\frac{\lambda_j^{\alpha}u_j}{\rho_{\alpha}})h_je_j,
 \qquad
 \sup_{u\in X^{\frac12}}
 \|\mathcal W'(u)\|_{\mathcal L(X^{\frac12})}<\infty.
 \end{equation}
\end{enumerate}
\end{lemma}
\begin{proof}
We prove the assertions separately. $(i)$
Let $u\in\mathcal A_1$. By \eqref{rhoalpaha},
$$
\|u\|_\alpha^2=\sum_{j=1}^{\infty}\lambda_j^{2\alpha}|u_j|^2
\leq\rho_\alpha^2.
$$
Therefore, for every $j\in\mathbb {Z}^{+}$,
$\lambda_j^{2\alpha}|u_j|^2\leq \rho_\alpha^2,$
and then
$|\frac{\lambda_j^\alpha u_j}{\rho_\alpha}|
\leq1.$
By the definition of $\eta$,
$\eta(
\frac{\lambda_j^\alpha u_j}{\rho_\alpha}
)=\frac{\lambda_j^\alpha u_j}{\rho_\alpha}.$
Hence,
$$
\rho_\alpha\lambda_j^{-\alpha}
\eta(\frac{\lambda_j^\alpha u_j}{\rho_\alpha}
)=u_j.
$$
It follows from \eqref{preparation-map} that
$$
\mathcal W(u)=\sum_{j=1}^{\infty}u_je_j=u.
$$

$(ii)$ For any $u,v\in X^{\frac12}$, the mean value theorem yields
$$
|\rho_\alpha\lambda_j^{-\alpha}
\eta(\frac{\lambda_j^\alpha u_j}{\rho_\alpha}
)-\rho_\alpha\lambda_j^{-\alpha}
\eta(\frac{\lambda_j^\alpha v_j}{\rho_\alpha}
)|\leq \eta_1|u_j-v_j|.
$$
Consequently,
$$
\begin{aligned}
\|\mathcal W(u)-\mathcal W(v)\|_{\frac12}^2
&=\sum_{j=1}^{\infty}
\lambda_j|\mathcal W_j(u)-\mathcal W_j(v)|^2
\\
&\leq \eta_1^2\sum_{j=1}^{\infty}
\lambda_j|u_j-v_j|^2\\
&=\eta_1^2\|u-v\|_{\frac12}^2,
\end{aligned}
$$
where $
\mathcal W_j(u)=\rho_\alpha\lambda_j^{-\alpha}
\eta(\frac{\lambda_j^\alpha u_j}{\rho_\alpha}).$ In addition, we note that $\eta(0)=0$, so that
$\mathcal W(0)=0.$
Thus,
$$
\|\mathcal W(u)\|_{\frac12}
\leq \eta_1\|u\|_{\frac12},
$$
which in particular shows that $\mathcal W$ is well defined on all of $X^{\frac12}$.

$(iii)$ For every $u\in X^{\frac12}$ and $s<\alpha-\frac12$, we have
$$
\begin{aligned}
\|\mathcal W(u)\|_s^2
&=\sum_{j=1}^{\infty}\lambda_j^{2s}|\mathcal W_j(u)|^2
\\
&\leq 4\rho_\alpha^2\sum_{j=1}^{\infty}\lambda_j^{-2(\alpha-s)}<\infty.
\end{aligned}
$$
Since $1<\alpha<\frac{3}{2},$ we may choose $s$ such that $\frac12<s<\alpha-\frac12,$ for this $s$, Sobolev embedding theorem gives
$$
X^s\hookrightarrow L^\infty(\Omega),\quad
X^s\hookrightarrow W^{1,p}(\Omega),\quad2<p<\frac{1}{1-s}.
$$
Hence, $(iii)$ holds.

$(iv)$ Let $u,h\in X^{\frac12}$ and $\varepsilon\neq0$. We have
\begin{align*}
\frac{\mathcal W(u+\varepsilon h)-\mathcal W(u)}{\varepsilon}=\sum_{j=1}^{\infty}D_{\varepsilon,j}e_j,
\end{align*}
where
\begin{align*}
D_{\varepsilon,j}:=\rho_\alpha\lambda_j^{-\alpha}
\frac{\eta(\frac{\lambda_j^\alpha(u_j+\varepsilon h_j)}{\rho_\alpha})
-\eta(\frac{\lambda_j^\alpha u_j}{\rho_\alpha})
}{\varepsilon}.
\end{align*}
For every fixed $j\in\mathbb {Z}^{+}$, by the differentiability of $\eta$, we obtain
\begin{align*}
D_{\varepsilon,j}
\longrightarrow
\eta'(\frac{\lambda_j^\alpha u_j}{\rho_\alpha})h_j,
\quad \text{as }\varepsilon\to0.
\end{align*}
Moreover, we obtain
\begin{align*}
\lambda_j
|D_{\varepsilon,j}-\eta'(
\frac{\lambda_j^\alpha u_j}{\rho_\alpha}
)h_j|^2\leq4\eta_1^2\lambda_j|h_j|^2.
\end{align*}
The dominated convergence theorem yields
\begin{align*}
&\|\frac{\mathcal W(u+\varepsilon h)-\mathcal W(u)}{\varepsilon}-\sum_{j=1}^{\infty}\eta'(
\frac{\lambda_j^\alpha u_j}{\rho_\alpha})h_je_j
\|_{\frac12}^2\\
&=\sum_{j=1}^{\infty}
\lambda_j|D_{\varepsilon,j}-\eta'(
\frac{\lambda_j^\alpha u_j}{\rho_\alpha}
)h_j|^2
\longrightarrow0,
\quad \text{as }\varepsilon\to0.
\end{align*}
Hence, $\mathcal W$ is Gateaux differentiable on $X^{\frac12}$ and
\begin{align*}
\mathcal W'(u)h=\sum_{j=1}^{\infty}
\eta'(\frac{\lambda_j^\alpha u_j}{\rho_\alpha})h_je_j.
\end{align*}
Moreover,
\begin{align*}
\|\mathcal W'(u)h\|_{\frac12}^2&=
\sum_{j=1}^{\infty}\lambda_j|\eta'(
\frac{\lambda_j^\alpha u_j}{\rho_\alpha})|^2
|h_j|^2
\\
&\leq\eta_1^2\sum_{j=1}^{\infty}
\lambda_j|h_j|^2=\eta_1^2\|h\|_{\frac12}^2.
\end{align*}
Therefore,
\begin{align*}
\sup_{u\in X^{\frac12}}
\|\mathcal W'(u)\|_{\mathcal L(X^{\frac12})}\leq \eta_1<\infty.
\end{align*}
\end{proof}

We now use the map $\mathcal{W}$ to modify the nonlinear term. Define
\begin{equation}\label{prepared-nonlinearity}
F(u):=\lambda f(\mathcal W(u)),\quad u\in X^{\frac12}.
\end{equation}

\begin{lemma}\label{prepared-nonlinearity-lemma}
The map ${F}:X^{\frac12}\to X^{\frac12}$ is globally bounded and globally Lipschitz continuous. More precisely, there exist constants $\tilde K,$ $\tilde{L}>0$ such that
\begin{equation}\label{tildeff}
	\begin{aligned}
	\|{F}(u)\|_{\frac12}&\leq\tilde K,\quad u\in X^{\frac12},\\
  \|{F}(u)-{F}(v)\|_{\frac12}
 &\leq \tilde{L}\|u-v\|_{\frac12}, \quad u,v\in X^{\frac12}.
\end{aligned}
\end{equation}
\end{lemma}
\begin{proof}
It follows from $(iii)$ of Lemma \ref{preparation-map-lemma} that there exist $p>2$ and $R>0$ such that
\begin{align}\label{bound1}
\|\nabla\mathcal{W}(u)\|_{L^{p}}+\|\mathcal W(u)\|_{L^\infty(\Omega)}\leq R,
\quad u\in X^{\frac12}.
\end{align}
Since $f\in C^2(\mathbb R)$, there exists a constant $C_{R}$ such that
\begin{align}\label{bound2}
\sup_{|s|\leq R} |f(s)|+|f^{\prime}(s)|+|f^{\prime\prime}(s)|\leq C_{R}.
\end{align}
We first prove the global boundedness of $F$. Since $f(0)=0$ and $\mathcal W(u)\in X^{\frac12}$, we have $f(\mathcal W(u))\in X^{\frac12}$. Moreover, there exists a constant $\tilde K>0$ such that
\begin{align*}
    \|{F}(u)\|_{\frac12}&=\lambda\|f(\mathcal{W}(u))\|_{\frac{1}{2}} \\ 
    &\leq \lambda\|\nabla f(\mathcal{W}(u))\| \\
    &=\lambda\|f^{\prime}(\mathcal{W}(u)) \nabla \mathcal{W}(u)\|\\
    &\leq\lambda C_R\|\nabla \mathcal{W}(u)\|\leq \tilde K.
\end{align*}
We next prove the global Lipschitz continuity. Let $q<\infty$ satisfy
\begin{align*}
\frac12=\frac1p+\frac1q.
\end{align*}
Since $p>2$, we have $q<\infty$, and Sobolev embedding theorem gives $X^{\frac12}\hookrightarrow L^q(\Omega).$
Thus, for $u,v\in X^{\frac12}$, it follows from 
\eqref{bound1}, \eqref{bound2} and the global Lipschitz continuity of $\mathcal{W}$ on $X^{\frac12}$ that 
\begin{align*}
\|f(\mathcal W(u))-f(\mathcal W(v))\|_{\frac12}
&=\|\nabla(f(\mathcal W(u))-f(\mathcal W(v)))\|\\
&=\|f'(\mathcal W(u))\nabla(\mathcal W(u)-\mathcal W(v))+(f'(\mathcal W(u))-f'(\mathcal W(v)))\nabla\mathcal W(v)\|\\
&\leq C_{R}\|\mathcal W(u)-\mathcal W(v)\|_{\frac12}+C_{R}\|\mathcal W(u)-\mathcal W(v)\|_{L^q(\Omega)}
\|\nabla\mathcal W(v)\|_{L^p(\Omega)}
\\
&\leq C_{R}(1+R)\|\mathcal W(u)-\mathcal W(v)\|_{\frac12}.
\end{align*}
Therefore, there exists a constant $\tilde L$ such that
\begin{align*}
\|F(u)-F(v)\|_{\frac12}\leq \tilde L\|u-v\|_{\frac12}, \quad u,v\in X^{\frac12}.
\end{align*}
This completes the proof.
\end{proof}

We now consider the following modified problem:
\begin{equation}\label{prepared-Kirchhoff}
 \left\{
 \begin{aligned}
  &\partial_tu+a(\|u\|_{\frac12}^2)Au=F(u),\quad t>0,\\
  &u(0)=u_0\in X^{\frac12},
 \end{aligned}
 \right.
\end{equation}
which can be regarded as a special case of the abstract model \eqref{nonlocal problem} by taking $\beta=0$, $\mathcal H=X^{\frac12}=H_0^1(\Omega)$, endowed with the inner product $(\cdot,\cdot)_{\frac12}$, and
$A=-\Delta:X^{\frac32}\subset X^{\frac12}\to X^{\frac12},$ subject to homogeneous Dirichlet boundary conditions. With respect to this Hilbert space structure, $A$ is positive self-adjoint with compact inverse. Moreover, if $\{e_j\}_{j=1}^{\infty}$ is the $L^2(\Omega)$ orthonormal eigenfunctions of $A$, then 
$\{\lambda_j^{-1/2}e_j\}_{j=1}^{\infty}$ is the orthonormal eigenfunctions of $X^{\frac12}$ with the same eigenvalues $\{\lambda_j\}_{j=1}^{\infty}$. By Lemma \ref{prepared-nonlinearity-lemma}, all assumptions on the nonlinear term in Section~2 are satisfied. Consequently, by applying a cut-off for the nonlocal coefficient $a(\|\nabla u\|^2)$ analogous to the one used above and applying Lemma~\ref{two square}, we conclude that the global attractor $\mathcal{A}_1$ of problem \eqref{2DKir} is contained in a Lipschitz graph over a suitable finite-dimensional subspace. This extends the corresponding one-dimensional result established in \cite{Xiaoqing} to the two-dimensional setting.

\noindent{\bf Data Availability} No datasets were generated or analyzed during the current study.\\ 

\noindent{\bf Declarations} The authors declare no conflict of interest.\\

\noindent{\bf Acknowledgements} The authors would like to thank the anonymous referee for helpful suggestions which helped to improve this paper. This work was carried out during the first author’s visit to the Institute of Mathematical and Computer Sciences (ICMC) at the University of São Paulo. She wishes to express her sincere gratitude to the members of ICMC for their warm hospitality and kindness. X. Yang was supported by China Scholarship Council File no. 202306180068. A. N. Carvalho was partially supported by Grants FAPESP
2020/14075-6 and CNPq 308902/2023-8. C. Sun was partially supported by NSFC 12271227.


\begin{thebibliography}{10}

\bibitem{Guo}
M.~Abu~Hamed, Y.~Guo, and E.~S. Titi.
\newblock Inertial manifolds for certain subgrid-scale {$\alpha$}-models of turbulence.
\newblock {\em SIAM J. Appl. Dyn. Syst.}, 14(3):1308--1325, 2015.

\bibitem{Estefani}
J.~M. Arrieta, A.~N. Carvalho, E.~M. Moreira, and J.~Valero.
\newblock Bifurcation and hyperbolicity for a nonlocal quasilinear parabolic problem.
\newblock {\em Adv. Differential Equations}, 29(1-2):1--26, 2024.

\bibitem{Brouwer}
L.~E.~J. Brouwer.
\newblock Zur {I}nvarianz des {$n$}-dimensionalen {G}ebiets.
\newblock {\em Math. Ann.}, 72(1):55--56, 1912.

\bibitem{valeroJDDE}
R.~Caballero, P.~Mar\'in-Rubio, and J.~Valero.
\newblock Existence and characterization of attractors for a nonlocal reaction-diffusion equation with an energy functional.
\newblock {\em J. Dynam. Differential Equations}, 34(1):443--480, 2022.

\bibitem{Carvalho-Langa-Robinson-13}
A.~N. Carvalho, J.~A. Langa, and J.~C. Robinson.
\newblock {\em Attractors for infinite-dimensional non-autonomous dynamical systems}, volume 182 of {\em Applied Mathematical Sciences}.
\newblock Springer, New York, 2013.

\bibitem{Carvalho2022}
A.~N. Carvalho, P.~Lappicy, E.~M. Moreira, and A.~N. Oliveira-Sousa.
\newblock A unified theory for inertial manifolds, saddle point property and exponential dichotomy.
\newblock {\em J. Differential Equations}, 416:1462--1495, 2025.

\bibitem{chipot2015}
M.~Chipot and T.~Savitska.
\newblock Asymptotic behaviour of the solutions of nonlocal {$p$}-{L}aplace equations depending on the {$L^p$} norm of the gradient.
\newblock {\em J. Elliptic Parabol. Equ.}, 1:63--74, 2015.

\bibitem{chipot2003}
M.~Chipot, V.~Valente, and G.~Vergara~Caffarelli.
\newblock Remarks on a nonlocal problem involving the {D}irichlet energy.
\newblock {\em Rend. Sem. Mat. Univ. Padova}, 110:199--220, 2003.

\bibitem{Chueshov}
I.~Chueshov.
\newblock {\em Dynamics of quasi-stable dissipative systems}.
\newblock Universitext. Springer, Cham, 2015.

\bibitem{constantin}
P.~Constantin, C.~Foias, B.~Nicolaenko, and R.~Temam.
\newblock {\em Integral manifolds and inertial manifolds for dissipative partial differential equations}, volume~70 of {\em Applied Mathematical Sciences}.
\newblock Springer-Verlag, New York, 1989.

\bibitem{FST88}
C.~Foias, G.~R. Sell, and R.~Temam.
\newblock Inertial manifolds for nonlinear evolutionary equations.
\newblock {\em J. Differential Equations}, 73(2):309--353, 1988.

\bibitem{Guo2018}
C.~G. Gal and Y.~Guo.
\newblock Inertial manifolds for the hyperviscous {N}avier-{S}tokes equations.
\newblock {\em J. Differential Equations}, 265(9):4335--4374, 2018.

\bibitem{Guo2024}
Y.~Guo.
\newblock Inertial manifolds for the two-dimensional hyperviscous navier-stokes equations.
\newblock arXiv:2401.14642, 2024.

\bibitem{Hale-ODE}
J.~K. Hale.
\newblock {\em Ordinary differential equations}.
\newblock Robert E. Krieger Publishing Co., Inc., Huntington, NY, second edition, 1980.

\bibitem{Henry}
D.~Henry.
\newblock {\em Geometric theory of semilinear parabolic equations}, volume 840 of {\em Lecture Notes in Mathematics}.
\newblock Springer-Verlag, Berlin-New York, 1981.

\bibitem{ANA2018}
A.~Kostianko.
\newblock Inertial manifolds for the 3{D} modified-{L}eray-{$\alpha $} model with periodic boundary conditions.
\newblock {\em J. Dynam. Differential Equations}, 30(1):1--24, 2018.

\bibitem{li2024}
A.~Kostianko, X.~Li, C.~Sun, and S.~Zelik.
\newblock Inertial manifolds via spatial averaging revisited.
\newblock {\em SIAM J. Math. Anal.}, 54(1):268--305, 2022.

\bibitem{Anna2023}
A.~Kostianko, C.~Sun, and S.~Zelik.
\newblock Inertial manifolds for 3{D} complex {G}inzburg-{L}andau equations with periodic boundary conditions.
\newblock {\em Indiana Univ. Math. J.}, 72(6):2403--2429, 2023.

\bibitem{Anna2015}
A.~Kostianko and S.~Zelik.
\newblock Inertial manifolds for the 3{D} {C}ahn-{H}illiard equations with periodic boundary conditions.
\newblock {\em Commun. Pure Appl. Anal.}, 14(5):2069--2094, 2015.

\bibitem{Li2020}
X.~Li and C.~Sun.
\newblock Inertial manifolds for the 3{D} modified-{L}eray-{$\alpha$} model.
\newblock {\em J. Differential Equations}, 268(4):1532--1569, 2020.

\bibitem{MS1988}
J.~Mallet-Paret and G.~R. Sell.
\newblock Inertial manifolds for reaction diffusion equations in higher space dimensions.
\newblock {\em J. Amer. Math. Soc.}, 1(4):805--866, 1988.

\bibitem{James2000}
J.~R. Munkres.
\newblock {\em Topology}.
\newblock Prentice Hall, Inc., Upper Saddle River, NJ, second edition, 2000.

\bibitem{Richards}
I.~Richards.
\newblock On the gaps between numbers which are sums of two squares.
\newblock {\em Adv. in Math.}, 46(1):1--2, 1982.

\bibitem{Robinson}
J.~C. Robinson.
\newblock {\em Infinite-dimensional dynamical systems}.
\newblock Cambridge Texts in Applied Mathematics. Cambridge University Press, Cambridge, 2001.
\newblock An introduction to dissipative parabolic PDEs and the theory of global attractors.

\bibitem{sell2002}
G.~R. Sell and Y.~You.
\newblock {\em Dynamics of evolutionary equations}, volume 143 of {\em Applied Mathematical Sciences}.
\newblock Springer-Verlag, New York, 2002.

\bibitem{Temam}
R.~Temam.
\newblock {\em Infinite-dimensional dynamical systems in mechanics and physics}, volume~68 of {\em Applied Mathematical Sciences}.
\newblock Springer-Verlag, New York, second edition, 1997.

\bibitem{Yagi}
A.~Yagi.
\newblock {\em Abstract parabolic evolution equations and their applications}.
\newblock Springer Monographs in Mathematics. Springer-Verlag, Berlin, 2010.

\bibitem{Xiaoqing}
X.~Yang, A.~N. Carvalho, and E.~M. Moreira.
\newblock Nonlocal kirchhoff-type parabolic equations in {$L^2$}: Reduction to finite dimension.
\newblock {\em Journal of Nonlinear Science}, 36:75, 2026.

\bibitem{Zelik}
S.~Zelik.
\newblock Inertial manifolds and finite-dimensional reduction for dissipative {PDE}s.
\newblock {\em Proc. Roy. Soc. Edinburgh Sect. A}, 144(6):1245--1327, 2014.

\end{thebibliography}

\end{document}